\documentclass[11pt, reqno]{amsart}
\usepackage{fancyhdr}
\usepackage{lipsum}
\usepackage{dsfont}
\usepackage{geometry}
\usepackage{amsmath,amsfonts,amssymb,mathrsfs,enumerate,amsopn,amsthm, graphicx}
\usepackage{hyperref}
\hypersetup{hypertex=true,
colorlinks=true,
linkcolor=black,
anchorcolor=black,
citecolor=black}
\allowdisplaybreaks[4]
\usepackage{enumitem}
\usepackage{color}
\usepackage{bm}
\usepackage{ctex}
\usepackage{epstopdf}
\usepackage{algorithmicx,enumerate,algorithm}
\usepackage{algpseudocode}
\usepackage{cite}
\usepackage{exscale}
\usepackage{relsize}
\ctexset{
abstractname={Abstract}
}
\numberwithin{equation}{section}
\newtheorem{theorem}{Theorem}[section]
\newtheorem{definition}{Definition}[section]
\newtheorem{lemma}{Lemma}[section]
\newtheorem{remark}{Remark}[section]

\usepackage{array}
\newlength{\onesixth}
\newtheoremstyle{noparens}%
  {}{}%
  {\itshape}{}%
  {\bfseries}{.}%
  { }%
  {\thmname{#1}\thmnumber{ #2}\mdseries\thmnote{ #3}}

\theoremstyle{noparens}
\newtheorem{lemmaNoParens}[lemma]{Lemma}

\makeatletter \@addtoreset{equation}{section} \makeatother

\begin{document}

\title{Local classical solutions to Navier-Stokes equations with  degenerate viscosities and vacuum}

\author{Yachun li}
\address[Y. Li]{School of Mathematical Sciences, MOE-LSC, and SHL-MAC, Shanghai Jiao Tong University, Shanghai 200240, P. R. China} \email{\tt ycli@sjtu.edu.cn}
\author{SHAOJUN YU}
\address[S. Yu]{School of Mathematical Sciences, Shanghai Jiao Tong University, Shanghai 200240, P. R. China}\email{\tt edwardsmith123@sjtu.edu.cn}
\maketitle
\begin{abstract}
We consider the 3D  isentropic compressible Navier-Stokes equations
with degenerate viscousities and vacuum. The degenerate viscosities $\mu(\rho)$ and $\lambda(\rho)$ are proportional to some power of density, while the powers of density in $\mu(\rho)$ and $\lambda(\rho)$ are different(i.e., $\delta_1\neq \delta_2$). The local
well-posedness of classical solution is established by introducing a ``quasi-symmetric hyperbolic''--``degenerate elliptic'' coupled structure to control the behavior of the velocity of the fluid near the vacuum and  give some uniform estimates.  In particular, the initial data allows vacuum in an open set and we do not need any initial compatibility conditions.
\end{abstract}
 
\section{introduction}
We consider the following three-dimensional isentropic Navier-Stokes equations:
\begin{equation}\label{1}
\left\{\begin{split}
&\rho_t+\operatorname{div}(\rho u)=0, \\
&(\rho u)_t+\operatorname{div}(\rho u \otimes u)+\nabla P=\operatorname{div} \mathbb{T},
\end{split}\right.
\end{equation}
with initial data
\begin{equation}\label{2}
(\rho,u)|_{t=0}=(\rho_0,u_0)(x), \quad x\in \mathbb{R}^3,
\end{equation}
and the far field behavior
\begin{equation}\label{3}
(\rho, u) \rightarrow(0,0) \quad \text { as } \quad|x| \rightarrow+\infty, \quad t \geq 0.
\end{equation}

In system \eqref{1}, $x \in \mathbb{R}^3, t \geq 0$ are space and time variables, respectively. $\rho(t,x)$ is the density, and $u(t,x) $ is the velocity of the fluid. Pressure $P$ satisfies
\begin{equation}
P(\rho)=A\rho^\gamma,\quad \gamma>1,
\end{equation}
where $A>0$ is a constant and $\gamma$ is the adiabatic exponent.

 $\mathbb{T}$ denotes the viscous stress tensor with the form
\begin{equation}
  \mathbb{T}=\mu(\rho)\left(\nabla u+(\nabla u)^{\top}\right)+\lambda(\rho) \operatorname{div} u \mathbb{I}_3,
\end{equation}
where $\mathbb{I}_3$ is the $3 \times 3$ identity matrix,
\begin{equation}\label{108}
\mu(\rho)= \alpha \rho^{\delta_1}, \quad \lambda(\rho)= \beta \rho^{\delta_2},
\end{equation}
$\mu(\rho)$ is the shear viscosity coefficient, $\lambda(\rho)+\frac{2}{3} \mu(\rho)$ is the bulk viscosity coefficient, $\alpha>0 $ and $\beta$ are both constants.

For the constant viscous fluid (i.e., $\delta_1=\delta_2=0$ in \eqref{108}), there is a lot of literature on the well-posedness of classical solutions. When $\inf_{x} \rho_0(x)>0$, the local well-posedness of  classical solutions  follows from the standard symmetric hyperbolic-parabolic structure satisfying the well-known Kawashima's condition (cf. \cite{Nash1962Le,Kawa1983Sys,   Serrin1959On}), which has been extended to a global one by Matsumura-Nishida \cite{Matsu1980The} near the nonvacuum equilibrium.
 In $\mathbb{R}$, the global well-posedness of strong solutions with arbitrarily large data in some bounded domains has been proven by Kazhikhov-Shelukhin \cite{Kazhikhov1977Unique}, and later, Kawashima-Nishida \cite{Kawashima1981Global} extended this theory to unbounded domains. 
When $\inf_{x} \rho_0(x)=0$, 
 the first main issue is the degeneracy of the time evolution operator, which makes it difficult to catch the behavior of the velocity field near the vacuum, for this case, the local-in-time well-posedness of strong solutions with vacuum was firstly solved by Cho-Choe-Kim\cite{Cho2004Unique} and Cho-Kim\cite{Cho2006On} in $\mathbb{R}^3$, where they introduced an initial compatibility condition to compensate the lack of a positive lower bound of density.
Later, Huang-Li-Xin \cite{Huang2012Glo} extended the local existence to a global one under some initial smallness assumption in $\mathbb{R}^3$. Jiu-Li-Ye \cite{Jiu2014Glo} proved the global existence of classical solution with arbitrarily large data and vacuum in $\mathbb{R}$.

For the degenerate viscous flow (i.e., $\delta_1> 0,\delta_2>0$ in \eqref{108}), they have received extensive attentions in recent years, especially for the case with
vacuum, where the well-posedness of classical solutions become more challenging due to the degenerate viscosity. In fact, the high order regularity estimates of the velocity in \cite{Cho2004Unique,Huang2012Glo,Jiu2014Glo} ($\delta=0$) strongly rely on the uniform ellipticity of the Lam\'e operator.
While for $\delta>0$, $\mu(\rho)$ vanish as the density function connects to vacuum continuously, thus it is difficult to adapt the approach of the constant viscosity case. A remarkable discovery of a new mathematical entropy function was made by Bresch-Desjardins \cite{Bre2003Exi}
for the viscosity satisfying some mathematical relation, 
 which provides additional regularity on some derivative of the density.
This observation was applied widely in proving the global existence of weak solutions with vacuum for \eqref{1} and some related models; see Bresch-Desjardins \cite{Bre2003Exi}, Bresch-Vasseur-Yu \cite{Bresch2022Glo}, Jiu-Xin \cite{Jiu2008The}, $\mathrm{Li}$-Xin \cite{Li2015Glo}, Mellet-Vasseur \cite{Mellet2007On}, Vasseur-Yu \cite{Vasseur2016Exi}, and so on. 
Moreover, 
when  $\delta=1$, Li-Pan-Zhu \cite{Li2017On} obtained the local existence of 2-D classical solution with far field vacuum, which also applies to the 2-D shallow water equations. 
When $1 < \delta \leq \min \left\{3, \frac{\gamma+1}{2}\right\}$, by making full use of the symmetrical structure of the hyperbolic operator and the weak smoothing effect of the elliptic operator, Li-Pan-Zhu \cite{Li2019On} established the local existence of 3-D classical solutions with arbitrarily large data and vacuum, see also Geng-Li-Zhu \cite{Geng2019Vanishing} for more related result, and Xin-Zhu \cite{Xin2021Glo} for the global existence of classical solution under some initial smallness assumptions in homogeneous Sobolev space. When $0<\delta<1$, Xin-Zhu \cite{Xin2021Well} obtained the local existence of 3-D local classical solution with far field vacuum, Cao-Li-Zhu\cite{Cao2022Glo} proved the global existence of 1-D classical solution with large initial data.   Some other interesting results and discussions can also be seen in Chen-Chen-Zhu\cite{Chen2022Vanishing}, Germain-Lefloch \cite{Germain2016Finite}, Guo-Li-Xin \cite{Guo2012Lagrange}, Lions \cite{Lions1998Math}, Luo-Xin-Zeng \cite{Luo2016Non}, Yang-Zhao \cite{Yang2002Vacuum}, Yang-Zhu \cite{Yang2002Com}, and so on.

It should be noted that the above literature on the well-posedness of the  3D Navier-Stokes equations with degenerate viscosities are all under the assumption that the powers of the viscosity coefficients  are the same, i.e.,($\delta_1=\delta_2$). While for Navier-Stokes equations with different powers of viscosities, there are few results, even for the local existence to 3D Navier-Stokes equations. Zhu \cite{Zhu2015Exi} established the local existence of classical solutions to 3D Navier-Stokes equations with $\mu(\rho)=\rho,$ and $\lambda(\rho)=\rho E(\rho)$, where $E(\rho)\in C^2(\bar{\mathbb{R}}^{+})$.  Luo-Zhou \cite{Luo2019On} obtained the local existence of classical solutions under the condition that $\mu(\rho)=\alpha\rho^{\delta_1},\lambda(\rho)=\beta\rho^{\delta_2}(0<\delta_1<1, \delta_1<\frac{\delta_1+1}{2}\leq \delta_2\leq 2\gamma-1)$
 with far field vacuum and compatibility conditions. In this paper, we established the local existence  and uniqueness of classical solutions to 3D Navier-Stokes equations under the condition that viscosities are the power law of density with different powers, satisfying that $\mu(\rho)=\alpha\rho^{\delta_1}, \lambda(\rho)=\beta \rho^{\delta_2}(\delta_1\neq \delta_2,\delta_1,\delta_2>1)$.  And the initial data can contain vacuum in an open set and without compatibility conditions.

Throughout this paper, we adopt the following simplified notations:
\begin{equation*}
\begin{aligned}
&|f|_p  =\|f\|_{L^p\left(\mathbb{R}^3\right)},\quad \|f\|_s=\|f\|_{H^s\left(\mathbb{R}^3\right)}, \quad D^{k, r} =\left\{f \in L_{l o c}^1\left(\mathbb{R}^3\right):\left|\nabla^k f\right|_r<+\infty\right\}, \\
& D^k =D^{k, 2},\quad |f|_{D^{k, r}}=\|f\|_{D^{k, r}\left(\mathbb{R}\right)},\quad |f|_{D^1}=\|f\|_{D^1\left(\mathbb{R}^3\right)},
 \quad \int f=\int_{\mathbb{R}^3} f \mathrm{~d} x,\\
&\|(f, g)\|_X=\|f\|_X+\|g\|_X, \quad \|f\|_{X\cap Y}=\|f\|_X+\|f\|_Y.
\end{aligned}
\end{equation*}

We first give the definition of regular solution.
\begin{definition}\label{87}
  Let $T>0$ be a finite constant. A solution $(\rho,u)$ to the Cauchy problem \eqref{1}-\eqref{3} is called a regular solution in $[0,T]\times \mathbb{R}^3$ if $(\rho,u)$ satisfies this problem in the sense of distribution  and:
  \begin{itemize}
\item[(1)]  $\rho \geq 0, \quad \rho^{\frac{\delta_1-1}{2}} \in C\left([0, T] ; H^3\right), \quad \rho^{\frac{\delta_2-1}{2}} \in C\left([0, T] ; H^3\right),\quad \rho^{\frac{\gamma-1}{2}} \in C\left([0, T] ; H^3\right);$
\item[(2)] $ u \in C([0, T] ; H^{s^{\prime}}) \cap L^{\infty}\left([0, T] ; H^3\right),  \quad \rho^{\frac{\delta_1-1}{2}}  \nabla^4 u \in L^2\left([0, T]; L^2\right); $
\item[(3)] $ u_t+u\cdot \nabla u=0$ {\rm as} $\rho(t, x)=0,$
\end{itemize}
where $s^{\prime} \in[2,3)$ is a constant.
\end{definition}

\begin{theorem}\label{38} {\rm (Local existence)} Let parameters $(\gamma, \delta_1,\delta_2,\alpha,\beta)$ satisfy
\begin{equation}\label{78}
\delta_2>\delta_1>1, \delta_2\geq\frac{5}{2}\delta_1-\frac{3}{2}, \gamma>1, \alpha>0, \min\{\delta_1, \gamma\}\leq 3.
\end{equation}
 If the initial data $\left(\rho_0, u_0\right)$ satisfy 
\begin{equation}\label{66}
\begin{split}
&\text{$(A_1)$} \quad \rho_0 \geq 0, \quad\left(\rho_0^{\frac{\gamma-1}{2}}, \rho_0^{\frac{\delta_1-1}{2}},  u_0\right) \in H^3,\\
&\text{$(A_2)$} \quad \rho_0^{\delta_2-\delta_1}\leq \frac{-\alpha}{3\beta}, \quad \text{when} \quad  \beta<0,
\end{split}
\end{equation}
then there exists a time $T_*>0$ and a unique regular solution $(\rho, u)$ in $\left[0, T_*\right] \times \mathbb{R}^3$ to the Cauchy problem \eqref{1}-\eqref{3} satisfying: 
\begin{equation*}
\begin{aligned}
\sup _{0 \leq t \leq T_*} & \left(\|\rho^{\frac{\gamma-1}{2}}\|_3^2+\|\rho^{\frac{\delta_1-1}{2}}\|_3^2 +\|u\|_3^2\right)(t)  + \int_0^t |\rho^{\frac{\delta_1-1}{2}}  \nabla^4 u|_2^2 \mathrm{d} s \leq C^0,
\end{aligned}
\end{equation*}
for arbitrary constant $s^{\prime} \in[2,3)$ and positive constant $C^0=C^0( \gamma, \delta_1, \delta_2, \rho_0, u_0)$. Actually, $(\rho, u)$ satisfies the Cauchy problem \eqref{1}-\eqref{3} classically in positive time $\left(0, T_*\right]$.
\end{theorem}

\begin{remark}
  The condition $\min\{\delta_1, \gamma\}\leq 3$ is only used to prove that the regular solution we obtained in Theorem \ref{38} is a classical one. Details can be seen in \S \ref{89}.
\end{remark}

\begin{remark}
  The condition $\text{$(A_2)$} $ is to ensure that the system meets the physical requirements that $2\mu(\rho)+3\lambda(\rho)\geq 0$ for a sufficiently small time $T_{*}$.
\end{remark}

We now comment on the analysis of this paper. By introducing two new quantities, equations \eqref{1}-\eqref{3} can be rewritten into a new system that consists of a transport equation for $\varphi$, and a ``quasi-symmetric hyperbolic"-``degenerate elliptic" coupled system with some special lower order source terms, which can control the behavior of the velocity of the fluid near the vacuum and  give some uniform estimates. The reason for setting $\delta_1<\delta_2$ is to ensure that $\alpha+\beta\varphi^{\frac{2(\delta_2-\delta_1)}{\delta_1-1}}>0.$ Since $\beta$ might be negative, then as $\varphi\rightarrow 0,$ if $\delta_1>\delta_2$, $\varphi^{\frac{2(\delta_2-\delta_1)}{\delta_1-1}}\rightarrow \infty.$ This would  prevent us from obtaining \eqref{92}, which is crucial for ensuring the elliptic structure necessary for subsequent energy estimates. 

Besides, we use the artificial viscosity coefficients $\eta^2>0$ in momentum equations, and the local well-posedness of linearized problem is established. Via passing to the limit as $\eta\rightarrow 0$, we obtain the solution of the linearized problem \eqref{20}, which allows that the elliptic operator appearing in the reformulated momentum equations is degenerate (Section \ref{83} ). Based on the uniform analysis for the linearized problem, we prove the local-in-time well-posedness of the nonlinear reformulated problem through the Picard iteration approach (Section \ref{109}).

We would like to clarify why the theorem imposes the restriction \(\delta_2 \geq \frac{5}{2} \delta_1 - \frac{3}{2}\), which implies that \(\frac{\delta_2 - \delta_1}{\delta_1 - 1} \geq \frac{3}{2}\) (where $\frac{\delta_2 - \delta_1}{\delta_1 - 1}$ is later defined as $m$). Since the presence of vacuum and there is no positive lower bound of the density, we cannot estimate the $L^\infty$ norm of $\frac{1}{\rho}.$ Through careful analysis and providing precise ranges for the parameters, we would be able to control the norm of the density.

The rest of this paper is organized as follows. \S \ref{90} is
dedicated for the preliminary lemmas to be used later.  \S \ref{91} is devoted to proving the local well-posedness of classical solutions to the Cauchy problem \eqref{1}-\eqref{3}, i.e., Theorem \ref{38}.

\section{preliminaries}\label{90}
The following well-known Gagliardo-Nirenberg inequality will be used.
\begin{lemmaNoParens} [\cite{Ladyzhenskaya1968}]\label{41} For $p \in[2,6], q \in(1, \infty)$, and $r \in(3, \infty)$, there exists some generic constant $C>0$ that may depend on $q$ and $r$ such that for
$$
f \in H^1\left(\mathbb{R}^3\right), \text { and } g \in L^q\left(\mathbb{R}^3\right) \cap D^{1, r}\left(\mathbb{R}^3\right),
$$
we have
$$
\begin{aligned}
& |f|_p^p \leq C|f|_2^{(6-p) / 2}|\nabla f|_2^{(3 p-6) / 2}, \\
& |g|_{\infty} \leq C|g|_q^{q(r-3) /(3 r+q(r-3))}|\nabla g|_r^{3 r /(3 r+q(r-3))} .
\end{aligned}
$$
Some special versions of this inequality can be written as
$$
|u|_6 \leq C|u|_{D^1}, \quad|u|_{\infty} \leq C\|\nabla u\|_1, \quad|u|_{\infty} \leq C\|u\|_{W^{1, r}}.
$$
\end{lemmaNoParens}

The second one will show some compactness results from the Aubin-Lions Lemma.
\begin{lemmaNoParens}
 [\cite{Simon1987compact}]\label{22}  Let $X_0, X$ and $X_1$ be three Banach spaces with $X_0 \subset X \subset X_1$. Suppose that $X_0$ is compactly embedded in $X$ and that $X$ is continuously embedded in $X_1$. Then:
\begin{itemize}
\item Let $G$ be bounded in $L^p\left([0, T]; X_0\right)$ where $1 \leq p<\infty$, and $\frac{\partial G}{\partial t}$ be bounded in $L^1\left([0, T]; X_1\right)$, then $G$ is relatively compact in $L^p([0, T]; X)$.
\item Let $F$ be bounded in $L^{\infty}\left([0, T]; X_0\right)$ and $\frac{\partial F}{\partial t}$ be bounded in $L^p\left([0, T]; X_1\right)$ with $p>1$, then $F$ is relatively compact in $C([0, T]; X)$.
\end{itemize}
\end{lemmaNoParens}

Next we give some Sobolev inequalities on the interpolation estimate, product estimate, composite function estimate and so on in the following  lemma:
\begin{lemmaNoParens}
 [\cite{Majda1986Compressible}]\label{103} Let $u \in H^s$, then for any $s^{\prime} \in[0, s]$, there exists a constant $C_s$ only depending on $s$ such that
\begin{equation*}
\|u\|_{s^{\prime}} \leq C_s|u|_2^{1-\frac{s^{\prime}}{s}}\|u\|_s^{\frac{s^{\prime}}{s}}.
\end{equation*}
\end{lemmaNoParens}

The following lemma is a useful tool for improving the weak convergence to a strong one.

\begin{lemmaNoParens}
 [\cite{Majda1986Compressible}]\label{104} If function sequence $\left\{w_n\right\}_{n=1}^{+\infty}$ converges weakly in a Hilbert space $X$ to $w$, then $w_n$ converges strongly to $w$ in $X$ if and only if
\begin{equation*}
\|w\|_X \geq \lim \sup _{n \rightarrow +\infty}\left\|w_n\right\|_X.
\end{equation*}
\end{lemmaNoParens}

The last lemma will be used to show the time continuity for the higher order terms of our solution:
\begin{lemmaNoParens}[\cite{Bol2003Semi}]\label{106}
 If $f(t, x) \in L^2\left([0, T] ; L^2\right)$, then there exists a sequence $s_k$ such that
$$
s_k \rightarrow 0, \quad \text { and } \quad s_k\left|f\left(s_k, x\right)\right|_2^2 \rightarrow 0, \quad \text { as } \quad k \rightarrow+\infty.
$$
\end{lemmaNoParens}

\section{Local-in-time well-posedness of solutions}\label{91}
\subsection{Reformulation}
By introducing two new quantities,
\begin{equation*}
 \phi= \rho^{\frac{\gamma-1}{2}},\quad    \varphi=\rho^{\frac{\delta_1-1}{2}},   
\end{equation*}
equations \eqref{1}-\eqref{3} can be rewritten into a new system that consists of a transport equations for $\varphi$, and a ``quasi-symmetric hyperbolic"-``degenerate elliptic" coupled system with some special lower order source terms for $(\phi, u)$:
\begin{equation}\label{4}
\left\{\begin{split}
&\varphi_t+u\cdot \nabla \varphi+\frac{\delta_1-1}{2} \varphi  \text{div}u=0, \\
&\underbrace{A_0 W_t+ \sum_{j=1}^{3}A_j(W)\partial_j W}_{\text {symmetric hyperbolic }}+\underbrace{\varphi^2 ( \alpha L_1(u)+(\alpha+\beta\varphi^{\frac{2(\delta_2-\delta_1)}{\delta_1-1}}) L_2(u))}_{\text {degenerate elliptic }}\\
=& \underbrace{\mathbb{Q}_1(W)H(\varphi)+\mathbb{Q}_2(W)H(\varphi^{\frac{\delta_2-1}{\delta_1-1}})}_{\text {lower order source }}   , \\
&(\varphi, W)|_{t=0}=\left(\varphi_0(x), W_0(x)\right), \quad x \in \mathbb{R}^3, \\
&(\varphi, W) \rightarrow(0,0), \quad \text { as } \quad|x| \rightarrow+\infty, \quad t \geq 0,
\end{split}\right.
\end{equation}
where $W=(\phi,u)^{\top}, W_0=(\phi_0,u_0)^{\top}$ and
\begin{align}
&A_0=
\left( \begin{array}{cc}
1 & 0 \\
0 & a_1 \mathbb{I}_3
 \end{array} \right),
 ~~~ A_j(W)=
\left( \begin{array}{cc}
u_j & \frac{\gamma-1}{2} \phi e_j^\top\\
\frac{\gamma-1}{2} \phi  e_j & a_1 u_j \mathbb{I}_3
 \end{array} \right),j=1,2,3, \nonumber\\ 
 &L_1(u)=
 \left( \begin{array}{c}
0 \\
-a_1\triangle u
 \end{array} \right),~~~
L_2(u)=
 \left( \begin{array}{c}
0 \\
-a_1\nabla \text{div}u
 \end{array} \right), 
 ~~~H(\varphi)=
 \left( \begin{array}{c}
0 \\
\nabla \varphi^2
 \end{array} \right),\nonumber\\
&\mathbb{Q}_1(W)=
\left( \begin{array}{cc}
0 &0 \\
0 & \frac{a_1\alpha \delta_1}{\delta_1-1} Q_1(u)
 \end{array} \right),
 ~~~\mathbb{Q}_2(W)=
\left( \begin{array}{cc}
0 &0 \\
0 & \frac{a_1\beta\delta_2}{\delta_2-1} Q_2(u)
 \end{array} \right),\nonumber\\
&Q_1(u)=\nabla u+(\nabla u)^\top,~~~ Q_2(u)=\text{div}u \mathbb{I}_3,~~~a_1=\frac{(\gamma-1)^2}{4A\gamma}.\nonumber
\end{align}

To prove Theorem \ref{38}, the first step is to establish the following existence of the strong solutions for the Cauchy problem \eqref{4}.

\begin{theorem}  \label{29}
 If initial data $\left(\varphi_0,\phi_0, u_0 \right)$ satisfies
\begin{equation}
\varphi_0 \geq 0, \quad \left(\varphi_0,\phi_0, u_0 \right) \in H^3,
\end{equation}
then there exists a time $T_*>0$ and a unique regular solution $(\varphi,\phi, u )$ in $\left[0, T_*\right] \times \mathbb{R}^3$ to the Cauchy problem \eqref{4} satisfying: 
\begin{equation}\label{88}
\begin{split}
&\varphi \in C([0,T];H^3), \varphi_t\in C([0,T];H^2), \phi\in C([0,T];H^3), \phi_t\in C([0,T];H^2), \\
&u\in C([0,T];H^{s^{\prime}})\cap L^{\infty}([0,T];H^3),s^\prime \in [2,3),
\varphi\nabla^4 u \in L^2\left([0, T] ; L^2\right),   \\
& \phi\nabla^4 u \in L^2\left([0, T] ; L^2\right),\quad u_t \in C\left([0, T] ; L^2\right) \cap L^2\left([0, T] ; D^1\right).
\end{split}
\end{equation}
Moreover, we have 
\begin{equation*}
\begin{aligned}
\sup _{0 \leq t \leq T_*} & \left(\|\phi\|_3^2+\|\varphi\|_3^2 +\|u\|_3^2\right)(t)  + \int_0^t |\varphi  \nabla^4 u|_2^2 \mathrm{d} s \leq C^0,
\end{aligned}
\end{equation*}
for arbitrary constant $s^{\prime} \in[2,3)$ and positive constant $C^0=C^0( \gamma, \delta_1, \delta_2 ,\varphi_0,\phi_0,  u_0)$.
\end{theorem}

\subsection{Linearization}
\begin{equation}\label{5}
\left\{\begin{split}
&\varphi_t+v\cdot \nabla \varphi+\frac{\delta_1-1}{2} \tilde{\varphi}  \text{div}v=0, \\
&A_0 W_t+ \sum_{j=1}^{3}A_j(V)\partial_j W+(\varphi^2+\eta^2)( \alpha L_1(u)+(\alpha+\beta\varphi^{\frac{2(\delta_2-\delta_1)}{\delta_1-1}}) L_2(u))\\
=&\mathbb{Q}_1(V)H(\varphi)+\mathbb{Q}_2(V)H(\varphi^{\frac{\delta_2-1}{\delta_1-1}})   , \\
&(\varphi, W)|_{t=0}=\left(\varphi_0(x), W_0(x)\right), \quad x \in \mathbb{R}^3, \\
&(\varphi, W) \rightarrow(0,0), \quad \text { as } \quad|x| \rightarrow+\infty, \quad t \geq 0,
\end{split}\right.
\end{equation}
where $\eta \in(0,1]$ is a constant, $W=(\phi, u)^{\top}, V=(\tilde{\phi}, v)^{\top}$ and $W_0=\left(\phi_0, u_0\right)^{\top}$. $(\tilde{\varphi}, \tilde{\phi})$ are both known functions and $v=(v_1,v_2,v_3)^\top\in \mathbb{R}^3$ is a known vector satisfying the initial assumption $(\tilde{\varphi},\tilde{\phi},v)(t=0, x)=\left(\varphi_0, \phi_0, u_0\right)(x)$ and
\begin{equation}\label{12}
\begin{split}
&\tilde{\varphi} \in C([0,T];H^3), \quad \tilde{\varphi}_t\in C([0,T];H^2), \quad \tilde{\phi}\in C([0,T];H^3), \quad \tilde{\phi}_t\in C([0,T];H^2), \\
&v\in C([0,T];H^{s^{\prime}})\cap L^{\infty}([0,T];H^3),s^\prime \in [2,3),\\
&\tilde{\varphi} \nabla^4 v \in L^2\left([0, T] ; L^2\right), \quad  v_t \in C\left([0, T] ; H^1\right) \cap L^2\left([0, T] ; D^2\right).
\end{split}
\end{equation}
Moreover, we assume that 
\begin{equation}\label{18}
\varphi_0\geq 0, \quad \varphi_0\in H^3.
\end{equation}

For simplicity, we denote a constant $m\triangleq\frac{(\delta_2-\delta_1)}{\delta_1-1}.$ And we define 
\begin{equation}\label{107}
Lu\triangleq -a_1(\alpha \triangle u+(\alpha+\beta\varphi^{2m})\nabla \text{div}u).
\end{equation}

First, we give the definition of high-dimensional elliptic operators, see \cite{Chen1998Sec} for details.
\begin{definition}
  Consider the following second order partial differential equations in $\mathbb{R}^d$ for $u=(u^{(1)},u^{(2)},\cdots, u^{(d)})$:
  \begin{equation}
    \sum_{p,q=1}^{d}\sum_{j=1}^{d}-\partial_p(A_{ij}^{pq}\partial_q u^{(j)})+\sum_{j=1}^{d}\partial_j f^{(j)}_i=0 \quad (i=1,2,\cdots, d),
  \end{equation}
where $1\leq i,j,p,q\leq d, f^{(j)}_i: \mathbb{R}^d \rightarrow \mathbb{R}$ are real functions. For a second-order partial differential operator $\Phi$, the ith component of $\Phi u$ is defined as $\sum\limits_{p,q=1}^{d}\sum\limits_{j=1}^{d}-\partial_p(A_{ij}^{pq}\partial_q u^{(j)}) (i=1,2,\cdots, d).$ If for any $\xi=(\xi_1,\xi_2,\cdots,\xi_d),\zeta=(\zeta_1,\zeta_2,\cdots,\zeta_d)\in \mathbb{R}^d,$ there exists a constant $\sigma>0,$ such that
\begin{equation}\label{9}
\sum_{p,q=1}^{d}\sum_{i,j=1}^{d}A_{ij}^{pq}\xi_p\xi_q\zeta_i\zeta_j\geq \sigma |\xi|^2|\zeta|^2,
\end{equation}
then  $\Phi$ is called an elliptic operator.
\end{definition}

Next, we will prove that $Lu$ defined in \eqref{107} is an elliptic operator.
\begin{lemma}
$Lu$ is an elliptic operator, i.e. for any $\xi=(\xi_1,\xi_2,\xi_3),\zeta=(\zeta_1,\zeta_2,\zeta_3)\in \mathbb{R}^3,$ there exists a constant $\sigma>0,$ such that
\begin{equation}\label{9}
\sum_{p,q=1}^{3}\sum_{i,j=1}^{3}A_{ij}^{pq}\xi_p\xi_q\zeta_i\zeta_j\geq \sigma |\xi|^2|\zeta|^2.
\end{equation}
\end{lemma}

\begin{proof}
From the definition of $Lu$, we have 
\begin{equation}
\begin{aligned}
&A_{11}^{11}=A_{22}^{22}=A_{33}^{33}=2a_1\alpha+a_1\beta\varphi^{2m},\\
&A_{11}^{22}=A_{11}^{33}=A_{22}^{11}=A_{22}^{33}=A_{33}^{11}=A_{33}^{22}=a_1\alpha,\\
&A_{12}^{12}=A_{13}^{13}=A_{21}^{21}=A_{23}^{23}=A_{31}^{31}=A_{32}^{32}=a_1\alpha+a_1\beta\varphi^{2m},
\end{aligned}
\end{equation}
$A_{ij}^{pq}=0,$ for others.

Therefore, for any $\xi=(\xi_1,\xi_2,\xi_3), \zeta=(\zeta_1,\zeta_2,\zeta_3)\in \mathbb{R}^3,$ we have
\begin{equation}
\begin{aligned}
\sum_{p,q=1}^{3}\sum_{i,j=1}^{3}A_{ij}^{pq}\xi_p\xi_q\zeta_i\zeta_j&=\sum_{i,j=1}^3\Big(a_1\alpha
\xi_j\xi_j\zeta_i\zeta_i+\big(a_1\alpha+a_1\beta\varphi^{2m}\big)\xi_i\xi_j\zeta_i\zeta_j\Big)\\
&=a_1\alpha|\xi|^2|\zeta|^2+ \sum_{i,j=1}^{3} a_1 \big(\alpha+\beta\varphi^{2m}\big)\xi_i\xi_j\zeta_i\zeta_j\\
&\triangleq \text{I}_1+\text{I}_2.
\end{aligned}
\end{equation}
\begin{equation}
\begin{aligned}
\text{I}_2&=a_1 \big(\alpha+\beta\varphi^{2m}\big)(\sum_{i=1}^{3}\xi_i\zeta_i)(\sum_{j=1}^{3}\xi_j\zeta_j)\\
&=a_1 \big(\alpha+\beta\varphi^{2m}\big)(\sum_{i=1}^{3}\xi_i\zeta_i)^2
\end{aligned}
\end{equation}
If $\beta<0$, since  $\alpha>0,$ and $\varphi_0^{2m}\leq -\frac{\alpha}{3\beta}$, there exists a $T^*$ such that 
\begin{equation*}
  \varphi(x,t)^{2m}\leq -\frac{\alpha}{2\beta},
\end{equation*}
which yields that
\begin{equation}\label{92}
  \alpha+\beta\varphi^{2m}\geq \frac{1}{2}\alpha> 0.
\end{equation}
And if $\beta\geq 0,$ \eqref{92} obviously holds.
Thus, $\text{I}_2\geq 0,$ which implies that
\begin{equation*}
  \text{I}_1+\text{I}_2\geq \text{I}_1\geq a_1\alpha |\xi|^2|\zeta|^2.
\end{equation*}
We can take $\sigma=a_1\alpha,$ and \eqref{9} holds. \qed
\end{proof}

\begin{lemma}\label{21}
 Assume that the initial data $\left(\varphi_0, \phi_0, u_0\right)$ satisfy \eqref{18}. Then there exists a unique strong solution $(\varphi, \phi, u)$ in $[0, T] \times \mathbb{R}^3$  to \eqref{5}  such that
$$
\begin{aligned}
& \varphi \in C\left([0, T] ; H^3\right),  \phi \in C\left([0, T] ; H^3\right), \\
& u \in C\left([0, T] ; H^3\right) \cap L^2\left([0, T] ; D^4\right), u_t \in C\left([0, T] ; H^1\right) \cap L^2\left([0, T] ; D^2\right),
\end{aligned}
$$
for some $T\in (0,T^*)$ and $\eta>0$.
\end{lemma}

\begin{proof} First, the existence and regularities of a unique solution $\varphi$ in $[0, T] \times \mathbb{R}^3$ to the equation $\eqref{5}_1$ can be obtained by the standard theory of transport equation (see \cite{Evans2010Partial}).

Second, when $\eta>0$, based on the regularities of $\varphi$, it is not difficult to solve $W$ from the linear symmetric hyperbolic-parabolic coupled system $\eqref{5}_2$ to complete the proof of this lemma (see \cite{Evans2010Partial}). Here we omit its details. \qed
\end{proof}

\subsection{A Priori estimates}\label{84}
We fix $T>0$ and a positive constant $c_0$ large enough such that
\begin{equation}\label{81}
1+\left\|\varphi_0\right\|_3+\left\|\phi_0\right\|_3+\left\|u_0\right\|_3 \leq c_0,
\end{equation}
and
\begin{equation}\label{26}
\begin{aligned}
& \sup _{0 \leq t \leq T^{**}}\left(\|\tilde{\varphi}(t)\|_1^2+\|\tilde{\phi}(t)\|_1^2+\|v(t)\|_1^2\right)+\int_0^{T^{**}} \left|\tilde{\varphi} \nabla^2 v\right|_2^2 \mathrm{d} t \leq c_1^2, \\
& \sup _{0 \leq t \leq T^{**}}\left(\|\tilde{\varphi}(t)\|_2^2+\|\tilde{\phi}(t)\|_2^2+\|v(t)\|_2^2\right)+\int_0^{T^{**}} \left|\tilde{\varphi} \nabla^3 v\right|_2^2 \mathrm{d} t \leq c_2^2, \\
& \sup _{0 \leq t \leq T^{**}}\left(\|\tilde{\varphi}(t)\|_3^2+\|\tilde{\phi}(t)\|_3^2+\|v(t)\|_3^2\right)+\int_0^{T^{**}} \left|\tilde{\varphi} \nabla^4 v\right|_2^2 \mathrm{d} t \leq c_3^2, \\
\end{aligned}
\end{equation}
for some time $T^{**} \in(0, T)$ and positive constants $c_i(i=1,2,3)$ such that
\begin{equation*}
c_0 \leq c_1 \leq c_2 \leq c_3 .
\end{equation*}
The constants $c_i(i=1,2,3)$ and $T^{**}$ will be determined later (see \eqref{82}) and depend only on $c_0$ and the fixed constants $\gamma, \delta_1,\delta_2,$ and $T$.

\begin{lemma}\label{23}
Let $T_1=\min \left\{T^{**},\left(1+c_3\right)^{-2}\right\}$.  Then for $0 \leq t \leq T_1$,
\begin{equation*}
\begin{aligned}
\|\varphi(t)\|_3\leq Cc_0, \quad \|\varphi_t(t)\|_2\leq C c_3^2. 
\end{aligned}
\end{equation*}
\end{lemma}
\begin{proof}
Apply the operator $\partial_x^\zeta(0\leq |\zeta|\leq 3)$ to $\eqref{5}_1,$ we have 
\begin{equation}\label{6}
  (\partial_x^\zeta \varphi)_t+v \cdot \nabla \partial_x^\zeta \varphi=-(\partial_x^\zeta(v\cdot \nabla \varphi)-v\cdot\nabla \partial_x^\zeta \varphi)-\frac{\delta_1-1}{2} \partial_x^\zeta(\tilde{\varphi}\text{div}v).
\end{equation}
Multiplying both sides of $\eqref{6}$ by $\partial_x^\zeta \varphi$, integrating over $\mathbb{R}^3$, we get
\begin{equation}
\begin{split}
   \frac{1}{2}\frac{\text{d}}{\text{d}t}|\partial_x^\zeta \varphi|_2^2\leq& C|\text{div} v|_\infty |\partial_x^\zeta \varphi|_2^2+C|\partial_x^\zeta(v\cdot \nabla \varphi-v\cdot\nabla \partial_x^\zeta \varphi)|_2|\partial_x^\zeta \varphi|_2 + C|\partial_x^\zeta(\tilde{\varphi}\text{div}v)|_2|\partial_x^\zeta \varphi|_2.
\end{split}
\end{equation}
First we consider the case where  $|\zeta|\leq 2.$
\begin{equation}
\begin{aligned}
 & |\partial_x^\zeta(v\cdot \nabla \varphi-v\cdot\nabla \partial_x^\zeta \varphi)|_2
  \leq C(|\nabla v\cdot \nabla \varphi|_2+|\nabla v\cdot \nabla^2 \varphi|_2+|\nabla^2 v\cdot \nabla \varphi|_2)\leq C\|v\|_3\|\varphi\|_2,\\
 & |\partial_x^\zeta(\tilde{\varphi}\text{div}v)|_2\leq C\|\tilde{\varphi}\|_2\|v\|_3,
\end{aligned}
\end{equation}
which yields that 
\begin{equation}
  \frac{\text{d}}{\text{d}t} \|\varphi\|_2\leq C\|v\|_3\|\varphi\|_2+C\|\tilde{\varphi}\|_2\|v\|_3.
\end{equation}
Then, according to Gronwall's inequality, one has 
\begin{equation}
   \|\varphi\|_2\leq (\|\varphi_0\|_2+Cc_3^2 t)\exp(Cc_3 t)\leq Cc_0,
\end{equation}
for $0\leq T_1=\min\{T^*,(1+c_3)^{-2}\}.$

When $|\zeta|= 3,$
\begin{equation*}
\begin{aligned}
|\partial_x^\zeta(v\cdot \nabla \varphi)-v\cdot\nabla \partial_x^\zeta \varphi|_2
\leq & C(|\nabla^3 v\cdot \nabla \varphi|_2+|\nabla^2 v\cdot \nabla^2 \varphi|_2+|\nabla v\cdot \nabla^3 \varphi|_2)\\
\leq & C\|v\|_3\|\varphi\|_3,
\end{aligned}
\end{equation*}
\begin{equation*}
\begin{aligned}
|\partial_x^\zeta(\tilde{\varphi}\text{div}v)|_2&\leq C(|\tilde{\varphi}\nabla^4 v|_2+|\nabla \tilde{\varphi}\cdot \nabla^3 v|_2+|\nabla^2 \tilde{\varphi}\nabla^2 v|_2 +|\nabla^3 \tilde{\varphi}\nabla v|_2)\\
&\leq C(|\tilde{\varphi}\nabla^4 v|_2+\|\tilde{\varphi}\|_3\|v\|_3).
\end{aligned}
\end{equation*}

Then we have 
\begin{equation}
\frac{\mathrm{d}}{\mathrm{d}t} \|\varphi\|_3\leq C(\|v\|_3\|\varphi\|_3+\|v\|_3\|\tilde{\varphi}\|_3+|\tilde{\varphi}\nabla^4 v|_2).
\end{equation}
According to Gronwall's Inequality, we obtain that
\begin{equation}\label{7}
\|\varphi\|_3\leq (\|\varphi_0\|_3+c_3^2 t+\int_0^t|\tilde{\varphi}\nabla^4 v|_2 \mathrm{d}s) \exp(Cc_3t).
\end{equation}
Observing that 
\begin{equation}\label{8}
\int_0^t |\tilde{\varphi}\nabla^4 v|_2\mathrm{d}s \leq t^{\frac{1}{2}}(\int_0^t |\tilde{\varphi}\nabla^4 v|_2^2 \mathrm{d}s )^{\frac{1}{2}}\leq c_3 t^{\frac{1}{2}}.
\end{equation}
It follows from \eqref{7}-\eqref{8} that 
\begin{equation}
\|\varphi\|_3\leq Cc_0,
\end{equation}
for $0\leq t\leq T_1.$ 

From 
\begin{equation}
\varphi_t=-v\cdot \nabla \varphi-\frac{\delta_1-1}{2} \tilde{\varphi}  \text{div}v ,
\end{equation}
we easily have 
\begin{equation*}
\begin{aligned}
|\varphi_t|_2\leq & C(|v|_\infty|\nabla\varphi|_2+|\tilde{\varphi}|_\infty|\text{div} v|_2)\leq Cc_2^2,\\
|\varphi_t|_{D^1}\leq & C(|v|_\infty|\nabla^2\varphi|_2+|\nabla v|_\infty|\nabla \varphi|_2+|\nabla \tilde{\varphi}|_{\infty}|\nabla v|_2+|\tilde{\varphi}|_\infty|\nabla^2 v|_2)\leq Cc_3^2,\\
|\varphi_t|_{D^2}\leq & C(|v|_\infty|\nabla^3\varphi|_2+|\nabla v|_\infty|\nabla^2 \varphi|_2+|\nabla^2 v|_2|\nabla \varphi|_\infty\\
&+|\nabla \tilde{\varphi}|_{\infty}|\nabla^2 v|_2+|\tilde{\varphi}|_\infty|\nabla^3 v|_2)\leq Cc_3^2.
\end{aligned}
\end{equation*}
Thus, we complete the proof of the lemma.
\qed
\end{proof}

\begin{lemma}
Let $(\varphi, W)$ be the solutions to \eqref{5}.  Then  
\begin{equation*}
\begin{aligned}
\|W(t)\|_1^2+\int_0^t |\sqrt{\varphi^2+\eta^2}\nabla^2 u|_2^2 \mathrm{d}s  \leq C c_0^2,
\end{aligned}
\end{equation*}
for $0 \leq t \leq T_2\triangleq \min \left\{T_1,\left(1+c_3\right)^{-4m-2}\right\}$.
\end{lemma}
\begin{proof}
Applying $\partial_x^\zeta$ to $\eqref{5}_2$, we have 
\begin{equation}\label{10}
\begin{split}
&A_0 \partial_x^\zeta W_t+\sum_{j=1}^{3} A_j(V)\partial_j \partial_x^\zeta W+ \alpha (\varphi^2+\eta^2) L_1(\partial_x^\zeta u)+(\varphi^2+\eta^2)(\alpha+\beta\varphi^{\frac{2(\delta_2-\delta_1)}{\delta_1-1}}) L_2(\partial_x^\zeta  u)\\
=& \partial_x^\zeta \mathbb{Q}_1(V)H(\varphi)+\partial_x^\zeta  \mathbb{Q}_2(V)H(\varphi^{\frac{\delta_2-1}{\delta_1-1}})-\sum_{j=1}^{3}\Big(\partial_x^\zeta\big(A_j(V)\partial_j W\big)-A_j(V)\partial_j \partial_x^\zeta W\Big)\\
& -\Big( \partial_x^\zeta\big( \alpha(\varphi^2+\eta^2) L_1(u) \big)- \alpha(\varphi^2+\eta^2) L_1(\partial_x^\zeta u) \Big)\\
&-\Big( \partial_x^\zeta\big( (\varphi+\eta^2)(\alpha+\beta\varphi^{\frac{2(\delta_2-\delta_1)}{\delta_1-1}}) L_2(u)\big)  -(\varphi+\eta^2)(\alpha+\beta\varphi^{\frac{2(\delta_2-\delta_1)}{\delta_1-1}}) L_2(\partial_x^\zeta u)  \Big)\\
 &+\Big( \partial_x^\zeta \big(\mathbb{Q}_1(V)H(\varphi)\big)- \partial_x^\zeta \mathbb{Q}_1(V)H(\varphi)\Big)\\
 &+\Big( \partial_x^\zeta \big(\mathbb{Q}_2(V)H(\varphi^{\frac{\delta_2-1}{\delta_1-1}})\big)- \partial_x^\zeta \mathbb{Q}_1(V)H(\varphi^{\frac{\delta_2-1}{\delta_1-1}})\Big).
\end{split} 
\end{equation}

Multiplying \eqref{10} by $\partial_x^\zeta W$ on both sides and integrating over $\mathbb{R}^3,$
\begin{align}
   &\frac{1}{2} \frac{\mathrm{d}}{\mathrm{d}t}\int (\partial_x^\zeta W)^\top A_0 \partial_x^\zeta W+a_1 \alpha|\sqrt{\varphi^2+\eta^2}\nabla \partial_x^\zeta u|_2^2+\frac{a_1\alpha}{2}|\sqrt{\varphi^2+\eta^2}\text{div}\partial_x^\zeta u|_2^2 \nonumber\\
   \leq&\frac{1}{2} \frac{\mathrm{d}}{\mathrm{d}t}\int (\partial_x^\zeta W)^\top A_0 \partial_x^\zeta W+a_1 \alpha|\sqrt{\varphi^2+\eta^2}\nabla \partial_x^\zeta u|_2^2+a_1\int (\alpha+\beta\varphi^{\frac{2(\delta_2-\delta_1)}{\delta_1-1}})(\varphi^2+\eta^2)\text{div}\partial_x^\zeta u\cdot \text{div}\partial_x^\zeta u\nonumber\\
  =& \frac{1}{2} \int (\partial_x^\zeta W)^\top \text{div}A(V) \partial_x^\zeta W -a_1\alpha\int(\nabla(\varphi^2+\eta^2)\cdot \nabla \partial_x^\zeta u )\cdot \partial_x^\zeta u\nonumber\\
  &-a_1 \int\Big(\nabla((\varphi^2+\eta^2)(\alpha+\beta\varphi^{\frac{2(\delta_2-\delta_1)}{\delta_1-1}})) \cdot \text{div} \partial_x^\zeta u \mathbb{I}_3 \Big)\cdot \partial_x^\zeta u \nonumber\\
  &+\frac{\alpha a_1\delta_1}{\delta_1-1}\int (\nabla \varphi^2 \partial_x^\zeta Q_1(v))\cdot \partial_x^\zeta u+\frac{\beta a_1\delta_2}{\delta_2-1}\int (\nabla \varphi^{\frac{2(\delta_2-1)}{\delta_1-1}}\partial_x^\zeta Q_2(v))\cdot \partial_x^\zeta u \nonumber\\
  &-\sum_{j=1}^{3}\int \Big(\partial_x^\zeta\big(\sum_{j=1}^{3}A_j(V)\partial_j W\big)-A_j(V)\partial_j \partial_x^\zeta W\Big) \cdot \partial_x^\zeta W \nonumber\\
  & -a_1\alpha\int \Big( \partial_x^\zeta\big( (\varphi^2+\eta^2) L_1(u) \big)- (\varphi^2+\eta^2) L_1(\partial_x^\zeta u) \Big)\partial_x^\zeta u \nonumber\\
  & -a_1 \int \Big( \partial_x^\zeta\big( (\varphi^2+\eta^2)(\alpha+\beta\varphi^{\frac{2(\delta_2-\delta_1)}{\delta_1-1}})  L_2(u) \big)- (\varphi^2+\eta^2)(\alpha+\beta\varphi^{\frac{2(\delta_2-\delta_1)}{\delta_1-1}})  L_2(\partial_x^\zeta u) \Big)\partial_x^\zeta u \nonumber\\
  &+\frac{\alpha a_1\delta_1}{\delta_1-1}\int \Big( \partial_x^\zeta \big(\nabla \varphi^2 \cdot  Q_1(v)\big)- \nabla \varphi^2\cdot  Q_1(\partial_x^\zeta v)\Big)\cdot \partial_x^\zeta u \nonumber\\
  &+\frac{\beta a_1\delta_2}{\delta_2-1}\int \Big( \partial_x^\zeta \big(\nabla \varphi^{\frac{2(\delta_2-1)}{\delta_1-1}}\cdot  Q_2(v)\big)- \nabla \varphi^{\frac{2(\delta_2-1)}{\delta_1-1}}\cdot  Q_2(\partial_x^\zeta v)\Big)\cdot \partial_x^\zeta u \nonumber \\
  \triangleq &\sum_{i=1}^{10}\text{I}_i,  \label{13}
\end{align}
where we have used \eqref{92} in the first inequality.

When $|\zeta|\leq 1,$
\begin{align}
\mathrm{I}_1&=\frac{1}{2} \int (\partial_x^\zeta W)^\top \text{div}A(V) \partial_x^\zeta W  \nonumber\\
  &\leq C|\nabla V|_\infty|\partial_x^\zeta W|_2^2 \leq C\|V\|_3 |\partial_x^\zeta W|_2^2 \nonumber\\
  &\leq C c_3 |\partial_x^\zeta W|_2^2, \nonumber\\
\mathrm{I}_2&=-a_1\alpha\int\big(\nabla(\varphi^2+\eta^2)\cdot \nabla \partial_x^\zeta u \big)\cdot \partial_x^\zeta u  \nonumber\\
   & \leq C|\nabla \varphi|_\infty |\varphi \nabla \partial_x^\zeta u|_2|\partial_x^\zeta u|_2 \nonumber\\
   & \leq \frac{a_1 \alpha}{10} |\sqrt{\varphi^2+\eta^2}\nabla \partial_x^\zeta u|_2^2+C c_0^2|\partial_x^\zeta u|_2^2, \nonumber\\
\mathrm{I}_3&=-a_1 \int\Big(\nabla((\varphi^2+\eta^2)(\alpha+\beta\varphi^{\frac{2(\delta_2-\delta_1)}{\delta_1-1}})) \cdot \text{div} \partial_x^\zeta u \mathbb{I}_3 \Big)\cdot \partial_x^\zeta u \nonumber\\
& \leq C (|\sqrt{\varphi^2+\eta^2}\text{div}\partial_x^\zeta u|_2|\nabla \varphi|_\infty|\alpha+\beta\varphi^{2m}|_\infty\nonumber\\
&\quad +|\varphi^2+\eta^2|_\infty|\varphi|_\infty^{2m-2}|\sqrt{\varphi^2+\eta^2}\text{div}\partial_x^\zeta u|_2)|\partial_x^\zeta u|_2 \nonumber\\
&\leq \frac{a_1 \alpha}{10} |\sqrt{\varphi^2+\eta^2} \text{div} \partial_x^\zeta u|_2^2+C c_0^{4m+2}|\partial_x^\zeta u|_2^2, \nonumber\\
\mathrm{I}_4&=\frac{\alpha a_1\delta_1}{\delta_1-1}\int (\nabla \varphi^2 \partial_x^\zeta Q_1(v))\cdot \partial_x^\zeta u \nonumber\\
  &\leq C|\varphi|_\infty|\nabla \varphi|_\infty|\nabla \partial_x^\zeta v|_2|\partial_x^\zeta u|_2 \nonumber\\
  &\leq Cc_3^3|\partial_x^\zeta u|_2\leq  Cc_3^3|\partial_x^\zeta u|_2^2+Cc_3^3, \nonumber\\
\mathrm{I}_5&=\frac{\beta a_1\delta_2}{\delta_2-1}\int (\nabla \varphi^{\frac{2(\delta_2-1)}{\delta_1-1}}\partial_x^\zeta Q_2(v))\cdot \partial_x^\zeta u \nonumber\\ 
  &\leq C|\varphi |_\infty^{2m+1}|\nabla\varphi |_\infty|\nabla^2 v|_2|\partial_x^\zeta u|_2 \nonumber\\
  &\leq Cc_3^{2m+3}|\partial_x^\zeta u |_2^2+Cc_3^{2m+3}, \nonumber\\
\mathrm{I}_6&=-\sum_{j=1}^{3}\int \Big(\partial_x^\zeta\big(\sum_{j=1}^{3}A_j(V)\partial_j W\big)-A_j(V)\partial_j \partial_x^\zeta W\Big) \cdot \partial_x^\zeta W \nonumber \\
  &\leq C|\nabla V|_\infty |\nabla W|_2^2\leq C c_3|\nabla W|_2^2, \nonumber \\
\mathrm{I}_7&=-a_1\alpha \Big( \partial_x^\zeta\big( (\varphi^2+\eta^2) L_1(u) \big)- (\varphi^2+\eta^2) L_1(\partial_x^\zeta u) \Big)\partial_x^\zeta u \\
  &\leq C|\nabla \varphi|_\infty|\varphi L_1(u)|_2|\partial_x^\zeta u|_2\leq \frac{a_1 \alpha}{10} |\sqrt{\varphi^2+\eta^2}\nabla^2 u|_2^2+C c_0^2|\partial_x^\zeta u|_2^2, \nonumber \\
\mathrm{I}_8&=-a_1 \int \Big( \partial_x^\zeta\big( (\varphi^2+\eta^2)(\alpha+\beta\varphi^{\frac{2(\delta_2-\delta_1)}{\delta_1-1}})  L_2(u) \big)- (\varphi^2+\eta^2)(\alpha+\beta\varphi^{\frac{2(\delta_2-\delta_1)}{\delta_1-1}})  L_2(\partial_x^\zeta u) \Big)\partial_x^\zeta u \nonumber \\
&\leq  C(|\nabla \varphi|_\infty|\alpha+\beta \varphi^{2m}|_\infty+|\varphi^2+\eta^2|_\infty|\varphi|_\infty^{2m-2}|\nabla \varphi|_\infty)|\varphi\nabla\text{div}u|_2|\nabla u|_2 \nonumber \\
&\leq \frac{a_1 \alpha}{10} |\sqrt{\varphi^2+\eta^2} \nabla \text{div}u|_2^2+C c_0^{4m+2}|\nabla u|_2^2, \nonumber \\
\mathrm{I}_{9}&=\frac{\alpha a_1\delta_1}{\delta_1-1}\int \Big( \partial_x^\zeta \big(\nabla \varphi^2 \cdot  Q_1(v)\big)- \nabla \varphi^2\cdot  Q_1(\partial_x^\zeta v)\Big)\cdot \partial_x^\zeta u \nonumber \\
&\leq C(|\varphi|_\infty|\nabla v|_\infty|\nabla^2 \varphi|_2+|\nabla v|_\infty|\nabla \varphi|_3|\nabla \varphi|_6)|\partial_x^\zeta u|_2  \nonumber \\
&\leq Cc_3^3 |\partial_x^\zeta u|_2 \leq Cc_3^3+Cc_3^3|\partial_x^\zeta u|_2^2, \nonumber \\
\mathrm{I}_{10}&=  \frac{\beta a_1\delta_2}{\delta_2-1}\int \Big( \partial_x^\zeta \big(\nabla \varphi^{\frac{2(\delta_2-1)}{\delta_1-1}}\cdot  Q_2(v)\big)- \nabla \varphi^{\frac{2(\delta_2-1)}{\delta_1-1}}\cdot  Q_2(\partial_x^\zeta v)\Big)\cdot \partial_x^\zeta u \nonumber \\
&\leq C(|\varphi|_\infty^{2m+1}|\nabla^2\varphi|_3|\nabla v|_6+|\varphi|_\infty^{2m}|\nabla \varphi|_\infty^2|\nabla v|_2)|\nabla u|_2  \nonumber \\
&\leq Cc_3^{2m+3} |\nabla u|_2^2+Cc_3^{2m+3}. \nonumber
\end{align}

Then, it yields that
\begin{equation}
\begin{split}
  &\frac{1}{2} \frac{\mathrm{d}}{\mathrm{d}t}\int (\partial_x^\zeta W)^\top A_0 \partial_x^\zeta W+ \frac{1}{2} a_1\alpha |\sqrt{\varphi^2+\eta^2}\nabla^2 u|_2^2\\
  \leq &Cc_3^{4m+2}\|W\|_1^2+Cc_3^{4m+2}.
  \end{split}
\end{equation}

According to Gronwall's inequality, we have 
\begin{equation}
\begin{split}
  &\|W\|_1^2+ \int_0^t |\sqrt{\varphi^2+\eta^2}\nabla^2 u|_2^2\mathrm{d}s \\
  \leq & C(\|W_0\|_1^2+c_3^{4m+2} t) \exp(Cc_3^{4m+2} t)\leq C c_0^2,
  \end{split}
\end{equation}
for $0\leq t\leq T_2=\min\{T_1,(1+c_3)^{-4m-2}\}$.

\end{proof}

\begin{lemma}
Let $(\varphi, W)$ be the solutions to \eqref{5}.  Then 
\begin{equation*}
\begin{aligned}
&|W(t)|_{D^2}^2+\int_0^t |\sqrt{\varphi^2+\eta^2}\nabla^3 u|_2^2 \mathrm{d}s   \leq C c_0^2,\\
&\|\phi_t\|_1+|u_t|_2+\int_0^t|u_t|_{D^1}^2 \mathrm{d}s\leq Cc_3^{4m+4},
\end{aligned}
\end{equation*}
for $0 \leq t \leq T_3\triangleq \min\{T_2,(1+c_3)^{-4m-4}\}$.
\end{lemma}

\begin{proof}
Now we consider the terms on the righthand side of \eqref{13} when $|\zeta|=2$. It follows from  lemma \ref{41}, \eqref{13},  H\"older's inequality and Young's inequality that 
\begin{align}
\mathrm{I}_1&=\frac{1}{2} \int (\partial_x^\zeta W)^\top \text{div}A(V) \partial_x^\zeta W  \nonumber\\
  &\leq C|\nabla V|_\infty|\partial_x^\zeta W|_2^2\leq C\|V\|_3 |\partial_x^\zeta W|_2^2 \nonumber\\
  &\leq C c_3 |\partial_x^\zeta W|_2^2, \nonumber\\
\mathrm{I}_2&=-a_1\alpha\int\big(\nabla(\varphi^2+\eta^2)\cdot \nabla \partial_x^\zeta u \big)\cdot \partial_x^\zeta u  \nonumber\\
   & \leq C|\nabla \varphi|_\infty |\varphi \nabla \partial_x^\zeta u|_2|\partial_x^\zeta u|_2 \nonumber\\
   & \leq \frac{a_1 \alpha}{10} |\sqrt{\varphi^2+\eta^2}\nabla \partial_x^\zeta u|_2^2+C c_0^2|\partial_x^\zeta u|_2^2, \nonumber\\
\mathrm{I}_3&=-a_1 \int\Big(\nabla((\varphi^2+\eta^2)(\alpha+\beta\varphi^{\frac{2(\delta_2-\delta_1)}{\delta_1-1}})) \cdot \text{div} \partial_x^\zeta u \mathbb{I}_3 \Big)\cdot \partial_x^\zeta u \nonumber\\
& \leq C \Big(|\sqrt{\varphi^2+\eta^2}\text{div}\partial_x^\zeta u|_2|\nabla \varphi|_\infty|\alpha+\beta\varphi^{2m}|_\infty\nonumber\\
&\quad +|\varphi^2+\eta^2|_\infty|\varphi|_\infty^{2m-2}|\sqrt{\varphi^2+\eta^2}\text{div}\partial_x^\zeta u|_2\Big)|\partial_x^\zeta u|_2 \nonumber\\
&\leq \frac{a_1 \alpha}{10} |\sqrt{\varphi^2+\eta^2} \text{div} \partial_x^\zeta u|_2^2+C c_0^{4m+2}|\partial_x^\zeta u|_2^2, \nonumber\\
\mathrm{I}_4&=\frac{\alpha a_1\delta_1}{\delta_1-1}\int (\nabla \varphi^2 \partial_x^\zeta Q_1(v))\cdot \partial_x^\zeta u \nonumber\\
  &\leq C|\varphi|_\infty|\nabla \varphi|_\infty|\nabla \partial_x^\zeta v|_2|\partial_x^\zeta u|_2 \nonumber\\
  &\leq Cc_3^3|\partial_x^\zeta u|_2\leq  Cc_3^3|\partial_x^\zeta u|_2^2+Cc_3^3, \nonumber\\
\mathrm{I}_5&=\frac{\beta a_1\delta_2}{\delta_2-1}\int (\nabla \varphi^{\frac{2(\delta_2-1)}{\delta_1-1}}\partial_x^\zeta Q_2(v))\cdot \partial_x^\zeta u \nonumber\\ 
  &\leq C|\varphi |_\infty^{2m+1}|\nabla\varphi |_\infty|\nabla^3 v|_2|\partial_x^\zeta u|_2 \nonumber\\
  &\leq Cc_3^{2m+3}|\partial_x^\zeta u |_2^2+Cc_3^{2m+3}, \nonumber\\
\mathrm{I}_6&=-\sum_{j=1}^{3}\int \Big(\partial_x^\zeta\big(\sum_{j=1}^{3}A_j(V)\partial_j W\big)-A_j(V)\partial_j \partial_x^\zeta W\Big) \cdot \partial_x^\zeta W \nonumber \\
  &\leq C|\nabla V|_\infty |\nabla W|_2^2\leq C c_3|\nabla W|_2^2, \nonumber \\
\mathrm{I}_7&=-a_1\alpha \Big( \partial_x^\zeta\big( (\varphi^2+\eta^2) L_1(u) \big)- (\varphi^2+\eta^2) L_1(\partial_x^\zeta u) \Big)\partial_x^\zeta u \nonumber \\
  &\leq C(|\nabla \varphi|_\infty^2|\nabla^2 u|_2^2+|\varphi |_3|\nabla^2 \varphi|_6|\nabla^2 u|_2^2+|\nabla \varphi|_\infty|\varphi \nabla^3 u|_2|\nabla^2 u|_2) \nonumber\\ 
  &\leq \frac{a_1 \alpha}{10} |\sqrt{\varphi^2+\eta^2}\nabla^3 u|_2^2+C c_0^2|\nabla^2 u|_2^2, \nonumber \\
\mathrm{I}_8&=-a_1 \int \Big( \partial_x^\zeta\big( (\varphi^2+\eta^2)(\alpha+\beta\varphi^{\frac{2(\delta_2-\delta_1)}{\delta_1-1}})  L_2(u) \big)- (\varphi^2+\eta^2)(\alpha+\beta\varphi^{\frac{2(\delta_2-\delta_1)}{\delta_1-1}})  L_2(\partial_x^\zeta u) \Big)\partial_x^\zeta u \nonumber \\
&\leq  C(|\nabla \varphi|_\infty|\alpha+\beta \varphi^{2m}|_\infty+|\sqrt{\varphi^2+\eta^2}|_\infty|\varphi|_\infty^{2m-1}|\nabla \varphi|_\infty)|\sqrt{\varphi^2+\eta^2} \nabla^2 \text{div}u|_2|\nabla^2 u|_2 \nonumber\\
&+|\nabla\varphi|_\infty^2|\varphi|_\infty^{2m}|\nabla ^2 u|_2^2+|\varphi|_\infty^{2m}|\nabla^2 \varphi|_3|\varphi \nabla^2 u|_6|\nabla^2 u|_2 \nonumber \\
&\leq \frac{a_1 \alpha}{10} |\sqrt{\varphi^2+\eta^2} \nabla^2 \text{div}u|_2^2+C c_0^{4m+2}|\nabla u|_2^2, \nonumber \\
\mathrm{I}_{9}&=\frac{\alpha a_1\delta_1}{\delta_1-1}\int \Big( \partial_x^\zeta \big(\nabla \varphi^2 \cdot  Q_1(v)\big)- \nabla \varphi^2\cdot  Q_1(\partial_x^\zeta v)\Big)\cdot \partial_x^\zeta u \nonumber \\
&\leq C(|\varphi|_\infty|\nabla v|_\infty|\nabla^2 \varphi|_2+|\nabla v|_\infty|\nabla \varphi|_3|\nabla \varphi|_6)|\partial_x^\zeta u|_2  \nonumber \\
&\leq Cc_3^3 |\partial_x^\zeta u|_2 \leq Cc_3^3+Cc_3^3|\partial_x^\zeta u|_2^2, \nonumber \\
\mathrm{I}_{10}&=  \frac{\beta a_1\delta_2}{\delta_2-1}\int \Big( \partial_x^\zeta \big(\nabla \varphi^{\frac{2(\delta_2-1)}{\delta_1-1}}\cdot  Q_2(v)\big)- \nabla \varphi^{\frac{2(\delta_2-1)}{\delta_1-1}}\cdot  Q_2(\partial_x^\zeta v)\Big)\cdot \partial_x^\zeta u \nonumber \\
&\leq C(|\varphi|_\infty^{2m}|\nabla \varphi|_\infty|\nabla^2 \varphi|_6|\nabla^2 v|_3|\nabla^2 u|_2+|\varphi|_\infty^{2m}|\nabla^3 \varphi|_2|\nabla^2 v|_3|\varphi \nabla^2 u|_6 \nonumber \\
&+|\varphi|_\infty^{2m-1}|\nabla \varphi|_\infty^3|\nabla^2 v|_2|\nabla^2 u|_2) \nonumber \\
&\leq Cc_3^{2m+3} |\nabla^2 u|_2^2+\frac{a_1\alpha}{10}|\sqrt{\varphi^2+\eta^2}\nabla^3 u|_2^2+Cc_3^{4m+4}. \nonumber
\end{align}

Then, it yields that
\begin{equation}
\begin{split}
  &\frac{1}{2} \frac{\mathrm{d}}{\mathrm{d}t}\int (\partial_x^\zeta W)^\top A_0 \partial_x^\zeta W+ \frac{1}{2} a_1\alpha |\sqrt{\varphi^2+\eta^2}\nabla^3 u|_2^2\\
  \leq &Cc_3^{4m+2}|W|_{D^2}^2+Cc_3^{4m+4}.
  \end{split}
\end{equation}

According to Gronwall's inequality, we have 
\begin{equation}
\begin{split}
  &|W|_{D^2}^2+ \int_0^t |\sqrt{\varphi^2+\eta^2}\nabla^3 u|_2^2\mathrm{d}s+\int_0^t|\sqrt{\varphi^2+\eta^2}\text{div} \partial_x^\zeta u|_2^2 \mathrm{d}s \\
  \leq & C(|W_0|_{D^2}^2+c_3^{4m+4} t) \exp(Cc_3^{4m+2} t)\leq C c_0^2,
  \end{split}
\end{equation}
for $0\leq t\leq T_3=\min\{T_2,(1+c_3)^{-4m-4}\}$.

From the relation that
\begin{equation}\label{14}
\phi_t=-v\cdot \nabla \phi-\frac{\gamma-1}{2}\tilde{\phi}\text{div} u,
\end{equation}
we have
\begin{equation}
\begin{split}
|\phi_t|_2&\leq C|v\cdot \nabla \phi+\tilde{\phi}\text{div} u|_2\\
&\leq C(|v|_\infty|\nabla \phi|_2+|\tilde{\phi}|_\infty|\text{div} u|_2)\\
&\leq Cc_2^2,\\
|\phi_t|_{D^1}&\leq C(|\nabla v|_\infty|\nabla \phi|_2+|v|_\infty|\nabla^2 \phi|_2+|\nabla \tilde{\phi}|_\infty|\nabla u|_2+|\tilde{\phi}|_\infty|\nabla^2 u|_2)\\
&\leq Cc_2^2.
\end{split}
\end{equation}

From the relation that
\begin{equation}\label{15}
\begin{split}
&u_t+v\cdot \nabla u+\frac{2A\gamma}{\gamma-1}\tilde{\phi}\nabla \phi-\alpha(\varphi^2+\eta^2)\triangle u-(\alpha+\beta\varphi^{2m})(\varphi^2+\eta^2)\nabla \text{div}u \\
=&\nabla\varphi^2\frac{\alpha\delta_1}{\delta_1-1}(\nabla v+(\nabla v)^\top)+\nabla\varphi^{2m+2}\frac{\beta\delta_2}{\delta_2-1}\text{div} v \mathbb{I}_3,
\end{split}
\end{equation}
we have
\begin{equation}
\begin{split}
|u_t|_2\leq& C (|v|_\infty|\nabla u|_2+|\tilde{\phi}|_\infty|\nabla \phi|_2+|\varphi^2+\eta^2|_\infty |\nabla^2 u |_2+|\varphi|_\infty^{2m+2}|\nabla^2 u|_2\\ &+|\varphi|_\infty|\nabla \varphi|_\infty|\nabla v|_2+|\varphi|_\infty^{2m+1}|\nabla \varphi|_\infty|\nabla v|_2)\\
 \leq & Cc_2^{2m+3}.
\end{split}
\end{equation}

For $|u_t|_{D^1}$,
\begin{equation}\label{16}
\begin{split}
|u_t|_{D^1}\leq &C(|\nabla v|_6 |\nabla u|_3+|v|_\infty|\nabla^2 u|_2+|\nabla \tilde{\phi}|_3|\nabla \phi|_6+|\tilde{\phi}|_\infty|\nabla^2 \phi|_2\\
&+|\sqrt{\varphi^2+\eta^2}|_\infty|\sqrt{\varphi^2+\eta^2}\nabla^3 u|_2(1+|\varphi|_\infty^{2m})+|\varphi|_\infty|\nabla \varphi|_\infty|\nabla^2 u|_2\\
&+(|\nabla \varphi|_\infty^2|\nabla v|_2+|\varphi|_\infty|\nabla^2 \varphi|_3|\nabla v|_6+|\varphi|_\infty|\nabla \varphi|_\infty|\nabla^2 v|_2)(1+|\varphi|_\infty^{2m}))\\
\leq & Cc_3^{2m+3}+Cc_3^{2m+1}|\sqrt{\varphi^2+\eta^2}\nabla^3 u|_2,
\end{split}
\end{equation}
which implies that
\begin{equation}
\begin{split}
&\int_{0}^{t} |u_t|_{D^1}^2 \mathrm{d}s\leq C\int_{0}^{t}(c_2^{4m+6}+c_3^{4m+2}|\sqrt{\varphi^2+\eta^2}\nabla^3 u|_2^2) \mathrm{d}s\\
\leq & Cc_3^{4m+4},
\end{split}
\end{equation}
for $0\leq t\leq T_2.$
\end{proof}

\begin{lemma}\label{17}
Let $(\varphi, W)$ be the solutions to \eqref{5}.  Then 
\begin{equation*}
\begin{aligned}
&|W(t)|_{D^3}^2+\int_0^t |\sqrt{\varphi^2+\eta^2}\nabla^4 u|_2^2 \mathrm{d}s \leq C c_0^2,\\
&|\phi_t|_{D^2}^2+|u_t|_{D^1}^2+\int_0^t|u_t|_{D^2}^2 \text{d}s\leq Cc_3^{4m+4},
\end{aligned}
\end{equation*}
for $0 \leq t \leq T_3$.
\end{lemma}
\begin{proof}
Now we consider the terms on the righthand side of \eqref{13} when $|\zeta|=3$. It follows from  lemma \ref{41}, \eqref{13},  H\"older's inequality and Young's inequality that 
\begin{align}
\mathrm{I}_1&=\frac{1}{2} \int (\partial_x^\zeta W)^\top \text{div}A(V) \partial_x^\zeta W  \nonumber\\
  &\leq C|\nabla V|_\infty|\partial_x^\zeta W|_2^2 \leq C\|V\|_3 |\partial_x^\zeta W|_2^2 \nonumber\\
  &\leq C c_3 |\partial_x^\zeta W|_2^2, \nonumber\\
\mathrm{I}_2&=-a_1\alpha\int\big(\nabla(\varphi^2+\eta^2)\cdot \nabla \partial_x^\zeta u \big)\cdot \partial_x^\zeta u  \nonumber\\
   & \leq C|\nabla \varphi|_\infty |\varphi \nabla \partial_x^\zeta u|_2|\partial_x^\zeta u|_2 \nonumber\\
   & \leq \frac{a_1 \alpha}{10} |\sqrt{\varphi^2+\eta^2}\nabla \partial_x^\zeta u|_2^2+C c_0^2|\partial_x^\zeta u|_2^2, \nonumber\\
\mathrm{I}_3&=-a_1 \int\Big(\nabla((\varphi^2+\eta^2)(\alpha+\beta\varphi^{\frac{2(\delta_2-\delta_1)}{\delta_1-1}})) \cdot \text{div} \partial_x^\zeta u \mathbb{I}_3 \Big)\cdot \partial_x^\zeta u \nonumber\\
& \leq C (|\sqrt{\varphi^2+\eta^2}\text{div}\partial_x^\zeta u|_2|\nabla \varphi|_\infty|\alpha+\beta\varphi^{2m}|_\infty \nonumber\\
&+|\varphi^2+\eta^2|_\infty|\varphi|_\infty^{2m-2}|\sqrt{\varphi^2+\eta^2}\text{div}\partial_x^\zeta u|_2)|\partial_x^\zeta u|_2) \nonumber\\
&\leq \frac{a_1 \alpha}{10} |\sqrt{\varphi^2+\eta^2} \text{div} \partial_x^\zeta u|_2^2+C c_0^{4m+2}|\partial_x^\zeta u|_2^2, \nonumber\\
\mathrm{I}_4&=\frac{\alpha a_1\delta_1}{\delta_1-1}\int (\nabla \varphi^2 \partial_x^\zeta Q_1(v))\cdot \partial_x^\zeta u  \label{H3}\\
  &\leq C|\varphi|_\infty|\nabla \varphi|_\infty|\nabla \partial_x^\zeta v|_2|\partial_x^\zeta u|_2 \nonumber\\
  &\leq Cc_3^3|\partial_x^\zeta u|_2\leq  Cc_3^3|\partial_x^\zeta u|_2^2+Cc_3^3, \nonumber\\
\mathrm{I}_5&=\frac{\beta a_1\delta_2}{\delta_2-1}\int (\nabla \varphi^{\frac{2(\delta_2-1)}{\delta_1-1}}\partial_x^\zeta Q_2(v))\cdot \partial_x^\zeta u \nonumber\\ 
  &\leq C(|\varphi |_\infty^{2m}|\nabla^2\varphi |_3|\nabla^3 v|_2|\varphi\nabla^3 u|_2+|\varphi|_\infty^{2m}|\nabla \varphi|_\infty^2|\nabla^3 v|_2|\nabla^3 u|_2+|\varphi|_\infty^{2m}|\nabla \varphi|_\infty|\nabla^3 v|_2|\varphi \nabla^4 u|_2) \nonumber\\
  &\leq \frac{a_1\alpha}{10}|\sqrt{\varphi^2+\eta^2}\nabla^4 u|_2^2+Cc_3^{2m+3}|\nabla^3 u|_2^2+Cc_3^{4m+4}, \nonumber\\
\mathrm{I}_6&=-\sum_{j=1}^{3}\int \Big(\partial_x^\zeta\big(\sum_{j=1}^{3}A_j(V)\partial_j W\big)-A_j(V)\partial_j \partial_x^\zeta W\Big) \cdot \partial_x^\zeta W \nonumber \\
  &\leq C|\nabla V|_\infty |\nabla W|_2^2\leq C c_3|\nabla W|_2^2, \nonumber \\
\mathrm{I}_7&=-a_1\alpha \Big( \partial_x^\zeta\big( (\varphi^2+\eta^2) L_1(u) \big)- (\varphi^2+\eta^2) L_1(\partial_x^\zeta u) \Big)\partial_x^\zeta u \nonumber \\
  &\leq C(|\nabla \varphi|_\infty|\nabla^2 \varphi|_3|\nabla^2 u|_6|\nabla^3 u|_2+|\nabla^3\varphi|_2|\nabla^2 u|_3|\varphi\nabla^3 u|_6+|\nabla \varphi|_\infty^2|\nabla^3 u|_2^2 \nonumber\\ 
  & \quad + |\nabla^2\varphi|_3|\varphi \nabla^3 u|_6|\nabla^3 u|_2+|\nabla\varphi|_\infty|\varphi \nabla^4 u|_2|\nabla^3 u|_2)       \nonumber\\ 
  &\leq \frac{a_1 \alpha}{10} |\sqrt{\varphi^2+\eta^2}\nabla^4 u|_2^2+C c_0^2|\nabla^3 u|_2^2, \nonumber \\
\mathrm{I}_8&=-a_1 \int \Big( \partial_x^\zeta\big( (\varphi^2+\eta^2)(\alpha+\beta\varphi^{\frac{2(\delta_2-\delta_1)}{\delta_1-1}})  L_2(u) \big)- (\varphi^2+\eta^2)(\alpha+\beta\varphi^{\frac{2(\delta_2-\delta_1)}{\delta_1-1}})  L_2(\partial_x^\zeta u) \Big)\partial_x^\zeta u \nonumber \\
&\leq  C(|\sqrt{\varphi^2+\eta^2}\nabla^4 u|_2|\nabla \varphi|_\infty|\varphi|_\infty^{2m}|\nabla^3 u|_2+|\nabla \varphi|_\infty^2|\varphi|_\infty^{2m}|\nabla^3 u|_2^2 \nonumber\\
&+|\nabla^2 \varphi|_3|\varphi|_\infty^{2m}|\varphi \nabla^3 u|_6|\nabla^3 u|_2+|\nabla^2 \varphi|_3|\varphi|_\infty^{2m-1}|\nabla \varphi|_\infty^2|\nabla^2 u|_6|\nabla^3 u|_2\nonumber\\
&+|\varphi|_\infty^{2m}|\nabla^3 \varphi|_2|\nabla^2 u|_3|\varphi \nabla^3 u|_6+|\varphi|_\infty^{2m}|\nabla \varphi|_\infty|\nabla^2 \varphi|_3|\nabla^2 u|_6| \nabla^3 u|_2\nonumber\\
&+|\varphi|_\infty^{2m-1}|\nabla \varphi|_\infty^2|\nabla \varphi|_3|\nabla^2 u|_6| \nabla^3 u|_2+|\varphi|_\infty^{2m-3}|\nabla^3 \varphi|_2|\sqrt{\varphi^2+\eta^2}|_\infty|\nabla^2 u|_3|\sqrt{\varphi^2+\eta^2} \nabla^3 u|_6)\nonumber\\
&\leq \frac{a_1\alpha}{10}|\sqrt{\varphi^2+\eta^2}\nabla^4 u|_2^2+Cc_0^{4m+2}|\nabla^3 u|_2^2, \nonumber\\
\mathrm{I}_{9}&=\frac{\alpha a_1\delta_1}{\delta_1-1}\int \Big( \partial_x^\zeta \big(\nabla \varphi^2 \cdot  Q_1(v)\big)- \nabla \varphi^2\cdot  Q_1(\partial_x^\zeta v)\Big)\cdot \partial_x^\zeta u \nonumber \\
&\leq C(|\varphi|_\infty|\nabla v|_\infty|\nabla^2 \varphi|_2+|\nabla v|_\infty|\nabla \varphi|_3|\nabla \varphi|_6)|\partial_x^\zeta u|_2  \nonumber \\
&\leq Cc_3^3 |\partial_x^\zeta u|_2 \leq Cc_3^3+Cc_3^3|\partial_x^\zeta u|_2^2, \nonumber \\
\mathrm{I}_{10}&=  \frac{\beta a_1\delta_2}{\delta_2-1}\int \Big( \partial_x^\zeta \big(\nabla \varphi^{\frac{2(\delta_2-1)}{\delta_1-1}}\cdot  Q_2(v)\big)- \nabla \varphi^{\frac{2(\delta_2-1)}{\delta_1-1}}\cdot  Q_2(\partial_x^\zeta v)\Big)\cdot \partial_x^\zeta u \nonumber \\
&\leq C(|\varphi|_\infty^{2m}|\nabla \varphi|_\infty^2|\nabla^3 v|_2|\nabla^3 u|_2+|\varphi|_\infty^{2m-2}|\nabla^2 \varphi|_3|\nabla^3 v|_6|\varphi \nabla^3 u|_6 \nonumber \\
&+|\varphi|_\infty^{2m-1}|\nabla \varphi|_\infty^2|\nabla \varphi|_3|\nabla^2 v|_6|\nabla^3 u|_2+|\varphi|_\infty^{2m}|\nabla \varphi|_\infty|\nabla^2 \varphi|_3|\nabla^2 v|_6|\nabla^3 u|_2 \nonumber \\
&+|\varphi|_\infty^{2m}|\nabla^3 \varphi|_2|\nabla^2 v|_3|\varphi \nabla^3 u|_6+|\varphi|_\infty^{2m-2}|\nabla \varphi|_\infty^3|\nabla \varphi|_2|\nabla  v|_\infty|\nabla^3 u|_2 \nonumber \\
&+|\varphi|_\infty^{2m-1}|\nabla \varphi|_\infty^2|\nabla^2 \varphi|_3|\nabla  v|_6|\nabla^3 u|_2+|\varphi|_\infty^{2m}|\nabla^2 \varphi|_3|\nabla^2 \varphi|_6|\nabla  v|_\infty|\nabla^3 u|_2\nonumber \\
&+|\varphi|_\infty^{2m}|\nabla \varphi|_\infty|\nabla^3 \varphi|_2|\nabla  v|_\infty|\nabla^3 u|_2) \nonumber \\
&\leq Cc_3^{2m+3} |\nabla^3 u|_2^2+\frac{a_1\alpha}{10}|\sqrt{\varphi^2+\eta^2}\nabla^4 u|_2^2+Cc_3^{4m+4}, \nonumber
\end{align}
where estimates in $\mathrm{I}_8$ use the condition that $m\geq\frac{3}{2}$, i.e.,$\delta_2\geq\frac{5}{2}\delta_1-\frac{3}{2}$ .

Then, it yields that
\begin{equation}
\begin{split}
  &\frac{1}{2} \frac{\mathrm{d}}{\mathrm{d}t}\int (\partial_x^\zeta W)^\top A_0 \partial_x^\zeta W+ \frac{1}{2} a_1\alpha |\sqrt{\varphi^2+\eta^2}\nabla^4 u|_2^2\\
  \leq &Cc_3^{2m+3}|W|_{D^3}^2+Cc_3^{4m+4}.
  \end{split}
\end{equation}

According to Gronwall's inequality, we have 
\begin{equation}
\begin{split}
  &|W|_{D^3}^2+ \int_0^t |\sqrt{\varphi^2+\eta^2}\nabla^4 u|_2^2\mathrm{d}s  \\
  \leq & C(|W_0|_{D^2}^2+c_3^{4m+4} t) \exp(Cc_3^{2m+3} t)\leq C c_0^2,
  \end{split}
\end{equation}
for $0\leq t\leq T_3=\min\{T_2,(1+c_3)^{-4m-4}\}$.

For $|\phi_t|_{D^2}$, from \eqref{14}, we have
\begin{equation}
\begin{split}
|\phi_t|_{D^2}\leq & C(|\nabla^2 v|_3|\nabla \phi|_6+|\nabla v|_3|\nabla^2\phi|_6+|v|_\infty|\nabla^3 \phi|_2\\
&+|\nabla^2\tilde{\phi}|_3|\nabla u|_6+|\nabla\tilde{\phi}|_3|\nabla^2 u|_6+|\tilde{\phi}|_\infty|\nabla^3 u|_2)\\
\leq &Cc_3^2.
\end{split}
\end{equation}

It follows from \eqref{16} that
\begin{equation}
\begin{split}
|u_t|_{D^1} \leq  Cc_3^{2m+2}+C|\sqrt{\varphi^2+\eta^2}\nabla^3 u|_2\leq  Cc_2^{2m+2}.
\end{split}
\end{equation}

For $|u_t|_{D^2}$, from \eqref{15}, we have
\begin{equation}
\begin{split}
|u_t|_{D^2}\leq & C\Big(\|v\|_3\|u\|_3(1+|\varphi|_\infty^{2m})+\|\tilde{\phi}\|_3\|\phi\|_3+\|\varphi\|_3\|u\|_3+\|\varphi\|_3^2\|u\|_3\\
&+|\varphi|_\infty|\sqrt{\varphi^2+\eta^2}\nabla^4 u|_2(1+|\varphi|_\infty^{2m})\Big)\\
\leq & Cc_3^{2m+2}+Cc_3^{2m+1}|\sqrt{\varphi^2+\eta^2}\nabla^4 u|_2,
\end{split}
\end{equation}
which implies that 
\begin{equation}
\begin{split}
\int_{0}^{t} |u_t|_{D^2}^2 \mathrm{d}s\leq &C\int_{0}^{t} \Big(c_3^{4m+4}+c_3^{4m+2}|\sqrt{\varphi^2+\eta^2}\nabla^4 u|_2^2\Big)\\
\leq & Cc_3^{4m+4},
\end{split}
\end{equation}
for $0\leq t\leq T_3.$

\end{proof}

Combining the estimates obtained in Lemmas \ref{23}-\ref{17}, we have 
\begin{equation}
\begin{split}
&\|\varphi\|_3^2\leq Cc_0^2,\quad \|\varphi_t\|_2\leq Cc_3^2, \\
& \|W\|_1^2+\int_{0}^{t}|\varphi\nabla^2 u|_2^2\mathrm{d}s  \leq Cc_0^2,\\
& |W|_{D^2}^2+\int_{0}^{t}|\varphi\nabla^3 u|_2^2\mathrm{d}s  \leq Cc_0^2,\\
&\|\phi_t\|_1^2+|u_t|_2^2+\int_{0}^{t}|u_t|_{D^1}^2 \mathrm{d}s \leq Cc_3^{4m+4},\\
&|W|_{D^3}^2+\int_{0}^{t}|\varphi\nabla^4 u|_2^2\mathrm{d}s  \leq Cc_0^2,\\
&\|\phi_t\|_{D^2}^2+|u_t|_{D^1}^2+\int_{0}^{t}|u_t|_{D^2}^2 \mathrm{d}s \leq Cc_3^{4m+4},
\end{split}
\end{equation}
for $0\leq t\leq T_3=\min\{T_2,(1+c_3)^{-4m-4}\}.$

Therefore, if we define the constants $c_i(i=1,2,3)$ and $T^{**}$ by 
\begin{equation}\label{82}
c_1=c_2=c_3=C^{\frac{1}{2}}c_0, \quad T^{**}=\min\{T,(1+c_3)^{-4m-4}\},
\end{equation}
then we deduce that 
\begin{equation}\label{20}
\begin{split}
&\sup_{0 \leq t\leq T^{**}}  \|\varphi(t)\|_1^2  +\|\phi(t)\|_1^2+\|u(t)\|_1^2+\int_{0}^{T^{**}}|\varphi \nabla^2 u|_2^2 \mathrm{d}t\leq c_1^2,\\
&\sup_{0 \leq t\leq T^{**}}  |\varphi(t)|_{D^2}^2  +\|\phi(t)\|_{D^2}^2+\|u(t)\|_{D^2}^2+\int_{0}^{T^{**}}|\varphi \nabla^3 u|_2^2 \mathrm{d}t\leq c_2^2,\\
&\sup_{0 \leq t\leq T^{**}}  |\varphi(t)|_{D^3}^2   +\|\phi(t)\|_{D^3}^2+\|u(t)\|_{D^3}^2+\int_{0}^{T^{**}}|\varphi \nabla^4 u|_2^2 \mathrm{d}t\leq c_3^2,\\
&\sup_{0 \leq t\leq T^{**}} \|u_t\|_1^2+\|\phi_t\|_2^2+\|\varphi_t\|_2^2+\int_{0}^{T^{**}}|u_t|_{D^2}^2\mathrm{d}t\leq c_3^{4m+4}.
\end{split}
\end{equation}

\subsection{Passing to the limit as $\eta\rightarrow 0$}\label{83}
Now we consider the following systems when $\eta\rightarrow 0$:
\begin{equation}\label{19}
\left\{\begin{split}
&\varphi_t+v\cdot \nabla \varphi+\frac{\delta_1-1}{2} \tilde{\varphi}  \text{div}v=0, \\
&A_0 W_t+ \sum_{j=1}^{3}A_j(V)\partial_j W+\varphi^2( \alpha L_1(u)+(\alpha+\beta\varphi^{\frac{2(\delta_2-\delta_1)}{\delta_1-1}}) L_2(u)) \\
=&\mathbb{Q}_1(V)H(\varphi)+\mathbb{Q}_2(V)H(\varphi^{\frac{\delta_2-1}{\delta_1-1}}), \\
&(\varphi, W)|_{t=0}=\left(\varphi_0(x), W_0(x)\right), \quad x \in \mathbb{R}^3, \\
&(\varphi, W) \rightarrow(0,0), \quad \text { as } \quad|x| \rightarrow+\infty, \quad t \geq 0.
\end{split}\right.
\end{equation}

\begin{lemma}\label{85} Assume $\left(\varphi_0, W_0\right)$ satisfy \eqref{18}. Then there exists a time $T^{**}>0$  and a unique strong solution $(\varphi,W)$ in $\left[0, T^{**}\right] \times \mathbb{R}^3$ to \eqref{19} such that
\begin{equation}
\begin{aligned}
& \varphi \in C\left(\left[0, T^{**}\right] ; H^3\right),  \quad  \phi \in C\left(\left[0, T^{**}\right] ; H^3\right), \\
& u \in C\left(\left[0, T^{**}\right] ; H^{s^{\prime}}\right) \cap L^{\infty}\left(\left[0, T^{**}\right] ; H^3\right), \quad s^{\prime} \in[2,3), \\
& \varphi \nabla^4 u \in L^2\left(\left[0, T^{**}\right] ; L^2\right), \quad u_t \in C\left(\left[0, T^{**}\right] ; H^1\right) \cap L^2\left(\left[0, T^{**}\right] ; D^2\right).
\end{aligned}
\end{equation}
 Moreover, $(\varphi, W)$ also satisfies the a priori estimates $\eqref{20}$.
\end{lemma}
\begin{proof}
We prove the existence, uniqueness and time continuity in three steps.

Step 1. Existence.  Due to Lemma \ref{21} and the uniform estimates \eqref{20}, for every $\eta>0$, there exists a unique strong solution $\left(\varphi^\eta, W^\eta\right)$ in $\left[0, T^{**}\right] \times \mathbb{R}^3$ to the linearized problem \eqref{5} satisfying estimates \eqref{20}, where the time $T^{**}>0$ is  independent of $\eta$.

By virtue of the uniform estimates \eqref{20} independent of $\eta$ and  Lemma \ref{22}, we know that for any $R>0$, there exists a subsequence of solutions (still denoted by) $\left(\varphi^\eta, W^\eta\right)$, which converges to a limit $(\varphi,  W)=(\varphi, \phi, u)$ in the following strong sense:
\begin{equation}\label{24}
\left(\varphi^\eta, W^\eta\right) \rightarrow(\varphi, W) \quad \text { in } C\left(\left[0, T^{**}\right] ; H^2\left(B_R\right)\right), \quad \text { as } \eta \rightarrow 0 .
\end{equation}

Again by virtue of the uniform estimates \eqref{20} independent of $\eta$, we also know that there exists a subsequence of solutions (still denoted by) $\left(\varphi^\eta, W^\eta\right)$, which converges to $(\varphi, W)$ in the following weak or weak* sense:
\begin{equation}\label{25}
\begin{aligned}
\left(\varphi^\eta,  W^\eta\right)\rightharpoonup(\varphi,  W) & \text { weakly* in } L^{\infty}\left(\left[0, T^{**}\right] ; H^3\left(\mathbb{R}^3\right)\right), \\
\left( \varphi_t^\eta, \phi_t^\eta\right)\rightharpoonup\left( \varphi_t, \phi_t \right) & \text { weakly* in } L^{\infty}\left(\left[0, T^{**}\right] ; H^2\left(\mathbb{R}^3\right)\right), \\
u_t^\eta \rightharpoonup u_t & \text { weakly* in } L^{\infty}\left(\left[0, T^{**}\right] ; H^1\left(\mathbb{R}^3\right)\right), \\
u_t^\eta \rightharpoonup u_t & \text { weakly in } L^2\left(\left[0, T^{**}\right] ; D^2\left(\mathbb{R}^3\right)\right),
\end{aligned}
\end{equation}
which, along with the lower semi-continuity of weak convergence, implies that $(\varphi,  W)$ also satisfies the corresponding estimates \eqref{20} except those of $\varphi \nabla^4 u$.

Combining the strong convergence in \eqref{24} and the weak convergence in \eqref{25}, we easily obtain that $(\varphi, W)$ also satisfies the local estimates \eqref{20} and
\begin{equation}\label{27}
\varphi^\eta \nabla^4 u^\eta\rightharpoonup\varphi \nabla^4 u \quad \text{weakly in}\quad  L^2\left(\left[0, T^{**}\right] \times \mathbb{R}^3\right).
\end{equation}

Now we want to show that $(\varphi, W)$ is a weak solution in the sense of distributions to the linearized problem \eqref{19}. Multiplying $\eqref{5}_2$ by test function $f(t, x)=$ $\left(f^1, f^2, f^3\right) \in C_c^{\infty}\left(\left[0, T^{**}\right) \times \mathbb{R}^3\right)$ on both sides, and integrating over $\left[0, T^{**}\right] \times \mathbb{R}^3$, we have
\begin{equation}\label{28}
\begin{aligned}
& \int_0^t \int_{\mathbb{R}^3} u^\eta \cdot f_t \mathrm{d} x \mathrm{d} s-\int_0^t \int_{\mathbb{R}^3}(v \cdot \nabla) u^\eta \cdot f \mathrm{d} x \mathrm{d} s-\int_0^t \int_{\mathbb{R}^3} \frac{2 A \gamma}{\gamma-1} \tilde{\phi} \nabla \phi^\eta f \mathrm{d} x \mathrm{d} s \\
=&\int_0^t \int_{\mathbb{R}^3} \alpha\Big((\varphi^\eta)^2+\eta^2\Big)L_1(u^\eta)  f \mathrm{d} x \mathrm{d} s+\int_0^t \int_{\mathbb{R}^3} (\varphi^\eta)^2(\alpha+\beta(\varphi^\eta)^{\frac{2(\delta_2-\delta_1)}{\delta_1-1}})L_2(u^\eta) f \mathrm{d} x \mathrm{d} s\\
&-\int_0^t \int_{\mathbb{R}^3}\Big(\nabla(\varphi^\eta)^2\cdot Q_1(v)+\nabla(\varphi^\eta)^{\frac{2(\delta_2-1)}{\delta_1-1}}\cdot Q_2(v)\Big) f \mathrm{d} x \mathrm{d} s -\int u_0  f(0, x).
\end{aligned}
\end{equation}

Combining the strong convergence in \eqref{24} and the weak convergences in \eqref{25}-\eqref{27},  and letting $\eta \rightarrow 0$ in \eqref{28}, we have
\begin{equation}
\begin{aligned}
& \int_0^t \int_{\mathbb{R}^3} u \cdot f_t \mathrm{d} x \mathrm{d} s-\int_0^t \int_{\mathbb{R}^3}(v \cdot \nabla) u \cdot f \mathrm{d} x \mathrm{d} s-\frac{2 A \gamma}{\gamma-1} \int_0^t \int_{\mathbb{R}^3} \tilde{\phi} \nabla \phi f \mathrm{d} x \mathrm{d} s \\
=&-\int u_0  f(0, x)+\int_0^t \int_{\mathbb{R}^3} \varphi^2 L_1(u) f \mathrm{d} x \mathrm{d} s+\int_0^t \int_{\mathbb{R}^3} \varphi^2(\alpha+\beta\varphi^{\frac{2(\delta_2-\delta_1)}{\delta_1-1}}) L_2(u) f \mathrm{d} x \mathrm{d} s\\
& -\int_0^t \int_{\mathbb{R}^3}\Big(\nabla \varphi^2\cdot Q_1(v)+\nabla \varphi^{\frac{2(\delta_2-1)}{\delta_1-1}}\cdot Q_2(v)\Big) f \mathrm{d} x \mathrm{d} s.
\end{aligned}
\end{equation}
Thus it is obvious that $(\varphi, W)$ is a weak solution in the sense of distributions to the linearized problem \eqref{19}, satisfying the regularities
\begin{equation}\label{86}
\begin{aligned}
& \varphi \in L^{\infty}\left(\left[0, T^{**}\right] ; H^3\right), \varphi_t \in L^{\infty}\left(\left[0, T^{**}\right] ; H^2\right), 
\phi \in L^{\infty}\left(\left[0, T^{**}\right] ; H^3\right), \\& \phi_t \in L^{\infty}\left(\left[0, T^{**}\right] ; H^2\right), u \in L^{\infty}\left(\left[0, T^{**}\right] ; H^3\right), \\
& \varphi \nabla^4 u \in L^2\left(\left[0, T^{**}\right] ; L^2\right), \quad u_t \in L^{\infty}\left(\left[0, T^{**}\right] ; H^1\right) \cap L^2\left(\left[0, T^{**}\right] ; D^2\right) .
\end{aligned}
\end{equation}

Step 2. Uniqueness. Let $\left(\varphi_1,  W_1\right)$ and $\left(\varphi_2, W_2\right)$ be two solutions obtained in the above step. We denote
\begin{equation*}
\bar{\varphi}=\varphi_1-\varphi_2,  \quad \bar{W}=W_1-W_2.
\end{equation*}
Then from $\eqref{19}_1$, we have
\begin{equation*}
\bar{\varphi}_t+v \cdot \nabla \bar{\varphi}=0, 
\end{equation*}
which implies that $\bar{\varphi}=0$.
Let $\bar{W}=(\bar{\phi}, \bar{u})^{\top}$, from $\eqref{19}_2$ and $\varphi_1=\varphi_2$, we have
\begin{equation}\label{42}
A_0 \bar{W}_t+ A_1(V)  \bar{W}_x=-\varphi_1^2(\alpha L_1(\bar{u})+(\alpha+\beta\varphi_1^{2m})L_2(\bar{u})).
\end{equation}
Then multiplying \eqref{42} by $\bar{W}$ on both sides, and integrating over $\mathbb{R}^3$, we have
\begin{equation}
\begin{aligned}
&\frac{1}{2}  \frac{\mathrm{d}}{\mathrm{d} t} \int \bar{W}^{\top} A_0 \bar{W}+a_1\alpha \left|\varphi_1  \nabla\bar{u}\right|_2^2 +a_1(\alpha+\beta \varphi_1^{2m}) \left|\varphi_1  \mathrm{div}\bar{u}\right|_2^2\\
 \leq & C(| \nabla V|_{\infty}|\bar{W}|_2^2+\left|\nabla \varphi_1\right|_{\infty}|\bar{u}|_2\left|\varphi_1 \nabla \bar{u}\right|_2 +|\varphi_1|_\infty^{2m-2}|\nabla \varphi_1|_\infty|\varphi_1\text{div}\bar{u}|_2|\bar{u}|_2) \\
 \leq & \frac{a_1 \alpha}{10}\left|\varphi_1  \nabla\bar{u}\right|_2^2+\frac{a_1 \beta}{10}\left|\varphi_1  \mathrm{div}\bar{u}\right|_2^2+Cc_2^{4m-2}|\bar{W}|_2^2,
\end{aligned}
\end{equation}
which yields that
\begin{equation}
  \frac{\mathrm{d}}{\mathrm{d} t} | \bar{W}|_2^2+\left|\varphi_1 \nabla\bar{u}\right|_2^2+\left|\varphi_1 \mathrm{div}\bar{u}\right|_2^2\leq Cc_2^2|\bar{W}|_2^2.
\end{equation}
From Gronwall's inequality, we obtain that $\bar{W}=0$, which gives the uniqueness.

Step 3. Time continuity. First for $\varphi$, via the regularities shown in \eqref{86} and the classical Sobolev imbedding theorem, we have
\begin{equation}\label{43}
\varphi \in C\left(\left[0, T^{**}\right] ; H^2\right) \cap C\left(\left[0, T^{**}\right] ; \text { weak }-H^3\right) .
\end{equation}
Using the same arguments as in Lemma \ref{23}, we have
\begin{equation*}
\|\varphi(t)\|_3^2 \leq\left(\left\|\varphi_0\right\|_3^2+C \int_0^t\left(\|\tilde{\varphi}\|_3^2\|v\|_3^2+\left|\tilde{\varphi}  \nabla^4 v\right|_2^2\right) \mathrm{d} s\right) \exp \left(C \int_0^t\|v\|_3 \mathrm{~d} s\right),
\end{equation*}
which implies that
\begin{equation*}
\lim \sup _{t \rightarrow 0}\|\varphi(t)\|_3 \leq\left\|\varphi_0\right\|_3 .
\end{equation*}
Then according to Lemma \ref{104} and \eqref{43}, we know that $\varphi$ is right continuous at $t=0$ in $H^3$ space. From the reversibility on the time to equation $\eqref{19}_1$, we know
\begin{equation}\label{44}
\varphi \in C\left(\left[0, T^{**}\right] ; H^3\right).
\end{equation}
For $\varphi_t$, from
\begin{equation*}
\varphi_t=-v \cdot \nabla \varphi-\frac{\delta-1}{2} \tilde{\varphi}  \text{div}v,
\end{equation*}
we only need to consider the term $\tilde{\varphi} \text{div} v$. Due to
\begin{equation*}
\tilde{\varphi}  \text{div}v \in L^2\left(\left[0, T^{**}\right] ; H^3\right), \quad(\tilde{\varphi}  \text{div}v)_t \in L^2\left(\left[0, T^{**}\right] ; H^1\right),
\end{equation*}
and the Sobolev imbedding theorem, we have
\begin{equation*}
\tilde{\varphi}  \text{div}v \in C\left(\left[0, T^{**}\right] ; H^2\right),
\end{equation*}
which implies that
\begin{equation*}
\varphi_t \in C\left(\left[0, T^{**}\right] ; H^2\right).
\end{equation*}
The similar arguments can be used to deal with the regularities of $ \phi$:
\begin{equation}
\phi\in C\left(\left[0, T^{**}\right] ; H^2\right), \quad  \phi_t \in C\left(\left[0, T^{**}\right] ; H^1\right).
\end{equation}
For velocity $u$, from the regularity shown in \eqref{86} and Sobolev's imbedding theorem, we obtain that
\begin{equation}\label{45}
u \in C\left(\left[0, T^{**}\right] ; H^1\right) \cap C\left(\left[0, T^{**}\right] ; \text {weak-}H^2\right) .
\end{equation}
Then from Lemma \ref{103}, for any $s^{\prime} \in[2,3)$, we have
\begin{equation*}
\|u\|_{s^{\prime}} \leq C|u|_2^{1-\frac{s^{\prime}}{2}}\|u\|_2^{\frac{s^{\prime}}{2}} \text {. }
\end{equation*}
Together with the upper bound shown in \eqref{20} and the time continuity \eqref{45}, we have
\begin{equation}\label{46}
u \in C(\left[0, T^{**}\right] ; H^{s^{\prime}}) .
\end{equation}
Finally, we consider $u_t$. From equations $\eqref{19}_2$ we have
\begin{equation*}
u_t=-v \cdot \nabla u-\frac{2A  \gamma}{\gamma-1} \tilde{\phi}\nabla \phi-a_1\varphi^2(\alpha  \triangle u +(\alpha+\beta\varphi^{2m})\nabla \text{div}u ) +\frac{\alpha\delta_1}{\delta_1-1}Q_1(u)\cdot \nabla\varphi^2+\frac{\beta\delta_2}{\delta_2-1}Q_2(u)\cdot \nabla\varphi^{2m+2}.
\end{equation*}
Then combining \eqref{86},\eqref{44},\eqref{46} and standard embedding theorem , we deduce that
\begin{equation*}
u_t \in C\left(\left[0, T^{**}\right] ; H^1\right) .
\end{equation*}  
The proof of Lemma \ref{85} is complete. \qed

\end{proof}

\subsection{Proof of Theorem \ref{29}} \label{109}

Our proof is based on the classical iteration scheme and the existence results for the linearized problem obtained in Section \ref{83}. Like in Section \ref{84}, we define constants $c_0$ and $c_i(i=1,2,3)$, and assume that
$$
1+\left\|\varphi_0\right\|_3+\left\|W_0\right\|_3 \leq c_0.
$$

Let $\left(\varphi^0, W^0=\left(\phi^0, u^0\right)\right)$, with the regularities
$$
\begin{aligned}
& \varphi^0 \in C\left(\left[0, T^{**}\right] ; H^3\right),  \quad \phi^0 \in C\left(\left[0, T^{**}\right] ; H^3\right), \quad \varphi^0 \nabla^4 u^0 \in L^2\left(\left[0, T^{**}\right] ; L^2\right),\\
& u^0 \in C(\left[0, T^{**}\right] ; H^{s^{\prime}}) \cap L^{\infty}\left(\left[0, T^{**}\right] ; H^3\right) \text { for any } s^{\prime} \in[2,3),
\end{aligned}
$$
be the solution to the problem
\begin{equation}
\left\{\begin{array}{l}
X_t+u_0 \cdot \nabla X=0 \quad \text { in }(0,+\infty) \times \mathbb{R}^3, \\
Y_t+u_0 \cdot \nabla Y=0 \quad \text { in }(0,+\infty) \times \mathbb{R}^3, \\
Z_t-X^2 \triangle Z=0 \quad \text { in }(0,+\infty) \times \mathbb{R}^3, \\
\left.(X, Y, Z)\right|_{t=0}=\left(\varphi_0, \phi_0, u_0\right) \quad \text { in } \mathbb{R}^3, \\
(X, Y, Z) \rightarrow(0,0,0) \quad \text { as }|x| \rightarrow+\infty, \quad t \geq 0 .
\end{array}\right.
\end{equation}

We take a time $T^{***} \in\left(0, T^{**}\right]$ small enough such that
\begin{equation}
\begin{array}{r}
\sup\limits_{0 \leq t \leq T^{***}}\left(|| \varphi^0(t)\left\|_1^2+|| \phi^0(t)\right\|_1^2+|| u^0(t) \|_1^2\right)+\int_0^{T^{***}} \left|\varphi^0 \nabla^2 u^0\right| \mathrm{d} t\leq c_1^2, \\
\sup\limits_{0 \leq t \leq T^{***}}\left(\left|\varphi^0(t)\right|_{D^2}^2+\left|\phi^0(t)\right|_{D^2}^2+\left|u^0(t)\right|_{D^2}^2\right)+\int_0^{T^{* **}} \left|\varphi^0 \nabla^3 u^0\right| \mathrm{d} t\leq c_2^2, \\
\sup\limits_{0 \leq t \leq T^{** *}}\left(\left|\varphi^0(t)\right|_{D^3}^2+\left|\phi^0(t)\right|_{D^3}^2+\left|u^0(t)\right|_{D^3}^2\right)+\int_0^{T^{***}} \left|\varphi^0 \nabla^4 u^0\right| \mathrm{d} t\leq c_3^2.
\end{array}
\end{equation}

Proof. We prove the existence, uniqueness and time continuity as follows. 

Step 1. Existence. Let $(\tilde{\varphi},\tilde{\phi},v)=\left(\varphi^0, \phi^0, u^0\right)$, we define $\left(\varphi^1, W^1\right)$ as a strong solution to problem \eqref{19}. Then we construct approximate solutions
\begin{equation*}
\left(\varphi^{k+1}, W^{k+1}\right)=\left(\varphi^{k+1},  \phi^{k+1}, u^{k+1}\right)
\end{equation*}
inductively, by assuming that $\left(\varphi^k, W^k\right)$ was defined for $k \geq 1$, let $\left(\varphi^{k+1},  W^{k+1}\right)$ be the unique solution to problem \eqref{19} with  $(\tilde{\varphi}, \tilde{\phi},v)$  replaced by $\left(\varphi^k,  W^k\right)$ as follows:
\begin{equation}\label{31}
\left\{\begin{split}
&\varphi^{k+1}_t+u^k\cdot \nabla \varphi^{k+1}+\frac{\delta_1-1}{2} \varphi^k  \text{div}u^k=0, \\
&A_0 W^{k+1}_t+ \sum_{j=1}^{3}A_j(W^k)\partial_j W^{k+1}+\alpha(\varphi^{k+1})^2 L_1(u^{k+1})+(\varphi^{k+1})^2(\alpha+\beta (\varphi^{k+1})^{2m}) L_2(u^{k+1})\\
=&\mathbb{Q}_1(W^k)\cdot H(\varphi^{k+1})+\mathbb{Q}_2(W^k)\cdot H((\varphi^{k+1})^{m+1})   , \\
&(\varphi^{k+1},W^{k+1})|_{t=0}=\left(\varphi_0(x), W_0(x)\right), \quad x \in \mathbb{R}^3, \\
&(\varphi^{k+1}, W^{k+1}) \rightarrow(0,0), \quad \text { as } \quad|x| \rightarrow+\infty, \quad t \geq 0.
\end{split}\right.
\end{equation}
It follows from Lemma \ref{85} that the sequence $\left(\varphi^k, W^k\right)$ satisfies the uniform a priori estimates \eqref{20} for $0 \leq t \leq T^{***}$.

Now we prove the convergence of the whole sequence $\left(\varphi^k, W^k\right)$ of approximate solutions to a limit $(\varphi, W)$ in some strong sense. Let
$$
\bar{\varphi}^{k+1}=\varphi^{k+1}-\varphi^k, \quad \bar{W}^{k+1}=\left(\bar{\phi}^{k+1}, \bar{u}^{k+1}\right)^{\top},
$$
with
$$
\bar{\phi}^{k+1}=\phi^{k+1}-\phi^k, \quad \bar{u}^{k+1}=u^{k+1}-u^k.
$$

Then, from \eqref{31}, we can obtain that
\begin{equation}\label{33}
\left\{\begin{aligned}
&\bar{\varphi}_t^{k+1}+u^k \cdot \nabla \bar{\varphi}^{k+1}+\bar{u}^k \cdot \nabla \varphi^k+\frac{\delta_1-1}{2}(\bar{\varphi}^k \operatorname{div} u^{k-1}+\varphi^k \operatorname{div} \bar{u}^k)=0, \\
&A_0 \bar{W}_t^{k+1}+\sum_{j=1}^3 A_j(W^k) \partial_j \bar{W}^{k+1}+\alpha(\varphi^{k+1})^2 L_1(\bar{u}^{k+1})+(\varphi^{k+1})^2(\alpha+\beta (\varphi^{k+1})^{2m})  L_2(\bar{u}^{k+1}) \\
=&\sum_{j=1}^3 A_j(\bar{W}^k) \partial_j W^k- \alpha \bar{\varphi}^{k+1}(\varphi^{k+1}+\varphi^k) L_1(u^k)- (\alpha+\beta (\varphi^{k+1})^{2m})L_2(u^k)(\varphi^{k+1}+\varphi^k)\bar{\varphi}^{k+1}\\
&-\beta(\varphi^k)^2L_2(u^k)((\varphi^{k+1})^{2m}-(\varphi^k)^{2m})+\mathbb{Q}_1(W^k)\cdot(\mathbb{H}(\varphi^{k+1})-\mathbb{H}(\varphi^k))  +\mathbb{Q}_1(\bar{W}^k)\cdot \mathbb{H}(\varphi^{k})\\
& +\mathbb{Q}_2(W^k)\cdot(\mathbb{H}((\varphi^{k+1})^{m+1})-\mathbb{H}((\varphi^k)^{m+1}))  +\mathbb{Q}_2(\bar{W}^k)\cdot \mathbb{H}((\varphi^{k})^{m+1}).
\end{aligned}\right.
\end{equation}
First, we consider $|\bar{\varphi}^{k+1}|_2$. Multiplying $\eqref{33}_1$ by $2 \bar{\varphi}^{k+1}$ and integrating over $\mathbb{R}^3$, one has
$$
\begin{aligned}
\frac{\mathrm{d}}{\mathrm{d}t}|\bar{\varphi}^{k+1}|_2^2= & -2 \int(u^k \cdot \nabla \bar{\varphi}^{k+1}+\bar{u}^k \cdot \nabla \varphi^k  +\frac{\delta_1-1}{2}\left(\bar{\varphi}^k \operatorname{div} u^{k-1}+\varphi^k \operatorname{div} \bar{u}^k\right) \bar{\varphi}^{k+1} \\
\leq & C|\nabla u^k|_{\infty}|\bar{\varphi}^{k+1}|_2^2+C|\bar{\varphi}^{k+1}|_2(|\bar{u}^k|_2|\nabla \varphi^k|_{\infty}  +|\bar{\varphi}^k|_2|\nabla u^{k-1}|_{\infty}+|\varphi^k \operatorname{div} \bar{u}^k|_2),
\end{aligned}
$$
which means that 
\begin{equation}\label{37}
\frac{\mathrm{d}}{\mathrm{d}t}|\bar{\varphi}^{k+1}(t)|_2^2 \leq C_\nu|\bar{\varphi}^{k+1}(t)|_2^2+\nu\left(|\bar{u}^k(t)|_2^2+|\bar{\varphi}^k(t)|_2^2+|\varphi^k \operatorname{div} \bar{u}^k(t)|_2^2\right)
\end{equation}
with $C_\nu=C\left(1+\nu^{-1}\right)$ and   $0<\nu \leq \frac{1}{10}$  is a constant.

Next, we consider $|\bar{W}^{k+1}|_2$. Multiplying $\eqref{33}_3$ by $2 \bar{W}^{k+1}$ and integrating over $\mathbb{R}^3$, we obtain that
\begin{equation}\label{35}
\begin{aligned}
&\frac{\mathrm{d}}{\mathrm{d}t}  \int(\bar{W}^{k+1})^{\top} A_0 \bar{W}^{k+1}+2 a_1  \alpha|\varphi^{k+1} \nabla \bar{u}^{k+1}|_2^2+a_1\alpha|\varphi^{k+1}\text{div}\bar{u}^{k+1}|_2^2 \\
\leq&  \int(\bar{W}^{k+1})^{\top} \operatorname{div} A(W^k) \bar{W}^{k+1}+\int \sum_{j=1}^3(\bar{W}^{k+1})^{\top} A_j(\bar{W}^k) \partial_j W^k \\
& -2 a_1\alpha   \int \nabla(\varphi^{k+1})^2 \cdot \nabla \bar{u}^{k+1} \cdot \bar{u}^{k+1} -2 a_1\beta  \int \nabla((\varphi^{k+1})^2(\alpha+\beta (\varphi^{k+1})^{2m}))  \mathrm{div} \bar{u}^{k+1} \cdot \bar{u}^{k+1}\\
& -2  \alpha \int(\bar{\varphi}^{k+1}(\varphi^{k+1}+\varphi^k)  L_1(u^k))\bar{u}^{k+1}-2   \int (\alpha+\beta (\varphi^{k+1})^{2m})L_2(u^k)(\varphi^{k+1}+\varphi^k)\bar{\varphi}^{k+1} \bar{u}^{k+1}  \\
&-2(2m-1)\beta \int (\varphi^k)^2L_2(u^k)(\theta^{k+1})^{2m-1}\bar{\varphi}^{k+1}\bar{u}^{k+1}\\
& -2 a_1\alpha \frac{\delta_1-1}{\delta_1} \int \nabla(\bar{\varphi}^{k+1}(\varphi^{k+1}+\varphi^k)) \cdot Q_1(u^k)\cdot \bar{u}^{k+1}\\
&-2(2m+2) a_1\beta \frac{\delta_2-1}{\delta_2}\int \nabla(\bar{\varphi}^{k+1}(\theta^{k+1})^{2m+1}) \cdot Q_2(u^k) \cdot \bar{u}^{k+1} \\
& +2 a_1 \alpha \frac{\delta_1-1}{\delta_1}\int \nabla(\varphi^k)^2 \cdot(  Q_1(u^k)-Q_1( u^{k-1})) \cdot \bar{u}^{k+1}\\
&+2 a_1 \beta\frac{\delta_2-1}{\delta_2} \int \nabla (\varphi^{k})^{2m+2} \cdot(Q_2(u^k)-Q_2( u^{k-1})) \cdot \bar{u}^{k+1}
:=\sum_{i=1}^{11} \mathrm{J}_i,
\end{aligned}
\end{equation}
where we used Lagrange’s mean value theorem and $\theta^{k+1}(x,t)$ is between $\varphi^{k+1}$ and $\varphi^k$.

It follows from  Lemma \ref{41}, \eqref{13},  H\"older's inequality and Young's inequality that 
\begin{align}
\mathrm{J}_1= & \int\left(\bar{W}^{k+1}\right)^{\top} \operatorname{div} A(W^k) \bar{W}^{k+1} \leq C|\nabla W^k|_{\infty}|\bar{W}^{k+1}|_2^2 \leq C|\bar{W}^{k+1}|_2^2, \nonumber\\
\mathrm{J}_2= & \int \sum_{j=1}^3 A_j\left(\bar{W}^k\right) \partial_j W^k \cdot \bar{W}^{k+1} \nonumber \\
\leq & C|\nabla W^k|_{\infty}|\bar{W}^k|_2|\bar{W}^{k+1}|_2 \leq C \nu^{-1}|\bar{W}^{k+1}|_2^2+\nu|\bar{W}^k|_2^2, \nonumber\\
\mathrm{J}_3= & -2 a_1\alpha   \int \nabla(\varphi^{k+1})^2 \cdot \nabla \bar{u}^{k+1} \cdot \bar{u}^{k+1} \nonumber\\
\leq & C |\nabla \varphi^{k+1}|_{\infty}|\varphi^{k+1} \nabla \bar{u}^{k+1}|_2|\bar{u}^{k+1}|_2 \label{34} \\
\leq & C |\bar{W}^{k+1}|_2^2+\frac{a_1  \alpha}{10}|\varphi^{k+1} \nabla \bar{u}^{k+1}|_2^2, \nonumber \\
\mathrm{J}_4=& -2 a_1\beta  \int \nabla((\varphi^{k+1})^2(\alpha+\beta (\varphi^{k+1})^{2m}))  \mathrm{div} \bar{u}^{k+1} \cdot \bar{u}^{k+1} \nonumber\\
\leq & C |\varphi^{k+1}\text{div}\bar{u}^{k+1}|_2|\bar{u}^{k+1}|_2|\nabla \varphi^{k+1}|_\infty|\varphi^{k+1}|_\infty^{2m} \nonumber\\
\leq & C |\bar{W}^{k+1}|_2^2+\frac{a_1  \alpha}{10}|\varphi^{k+1} \mathrm{div} \bar{u}^{k+1}|_2^2, \nonumber\\
\mathrm{J}_5= & -2  \alpha \int(\bar{\varphi}^{k+1}(\varphi^{k+1}+\varphi^k) \cdot L_1(u^k))\cdot \bar{u}^{k+1}\nonumber\\
\leq & C|\bar{\varphi}^{k+1}|_2|\bar{u}^{k+1}|_2|\varphi^k \nabla^2 u^k|_\infty+C|\bar{\varphi}^{k+1}|_2|\varphi^{k+1}\nabla \bar{u}^{k+1}|_2 \nonumber\\
\leq & C|\bar{\varphi}^{k+1}|_2^2+C(1+|\varphi^{k+1}\nabla^4 u^k|_2)|\bar{u}^{k+1}|_2^2+\frac{a_1\alpha}{10}|\varphi^{k+1}\nabla \bar{u}^{k+1}|_2^2, \nonumber\\
\mathrm{J}_6= & -2   \int (\alpha+\beta (\varphi^{k+1})^{2m})L_2(u^k)(\varphi^{k+1}+\varphi^k)\bar{\varphi}^{k+1} \bar{u}^{k+1}\nonumber\\
\leq & C|\bar{\varphi}^{k+1}|_2|\bar{u}^{k+1}|_2|\varphi^k \nabla \text{div} u^k|_\infty+C|\bar{\varphi}^{k+1}|_2|\varphi^{k+1}\bar{u}^{k+1}|_3|\nabla \text{div} u^k|_6   \nonumber\\
\leq & C|\bar{\varphi}^{k+1}|_2^2+C(1+|\varphi^{k}\nabla^3\text{div} u^k|_2^2)|\bar{u}^{k+1}|_2^2+\frac{a_1\alpha}{10}|\varphi^{k+1}\nabla\bar{u}^{k+1}|_2^2 , \nonumber\\
\mathrm{J}_7=& -2(2m-1)\beta \int (\varphi^k)^2L_2(u^k)(\theta^{k+1})^{2m-1}\bar{\varphi}^{k+1}\bar{u}^{k+1} \nonumber\\
\leq & C(1+|\varphi^k\nabla^3 \text{div}u^k|_2)|\bar{W}^{k+1}|_2^2+C|\varphi^{k+1}|_2^2, \nonumber\\
\mathrm{J}_8=& -2 a_1\alpha  \frac{\delta_1-1}{\delta_1}\int \nabla(\bar{\varphi}^{k+1}(\varphi^{k+1}+\varphi^k)) \cdot Q_1(u^k) \cdot \bar{u}^{k+1}\nonumber\\
\leq & C|\nabla^2 u^k|_6|\bar{\varphi}^{k+1}|_2|\varphi^{k+1}\bar{u}^{k+1}|_3+C|\bar{\varphi}^{k+1}|_2|\varphi^k\nabla^2u^k|_\infty |\bar{u}^{k+1}|_2 \nonumber\\
&+C|\bar{\varphi}^{k+1}|_2|\nabla u^k|_\infty |\varphi^{k+1}\nabla \bar{u}^{k+1}|_2+C|\bar{\varphi}^{k+1}|_2|\varphi^{k+1}\nabla \bar{u}^{k+1}|_2|\nabla u^k|_\infty \nonumber\\
&+C |\bar{\varphi}^{k+1}|_2^2 |\nabla u^k|_\infty|\nabla \bar{u}^{k+1}|_\infty \nonumber\\
\leq & C |\bar{\varphi}^{k+1}|_2^2+C(1+|\varphi^k\nabla^4 u^k|_2^2)|\bar{W}^{k+1}|_2^2+\frac{a_1\alpha}{10}|\varphi^{k+1}\nabla \bar{u}^{k+1}|_2^2,\nonumber\\
\mathrm{J}_9=& -2(2m+2) a_1\beta \frac{\delta_2-1}{\delta_2}\int \nabla(\bar{\varphi}^{k+1}(\theta^{k+1})^{2m+1}) \cdot Q_2(u^k) \cdot \bar{u}^{k+1} \nonumber\\
\leq & C (|\bar{\varphi}^{k+1}|_2|\varphi^k \nabla \mathrm{div} u^k|_\infty|\varphi^k|_\infty^m |\bar{u}^{k+1}|_2+|\bar{\varphi}^{k+1}|_2|\varphi^{k+1} \bar{u}^{k+1}|_6|\nabla \text{div}u^k|_3|\varphi^k|_\infty^m \nonumber\\
&+|\bar{\varphi}^{k+1}||\varphi^{k+1}\nabla \bar{u}^{k+1}|_2|\text{div}u^k|_\infty+|\bar{\varphi}^{k+1}|^2|\nabla \bar{u}^{k+1}|_\infty|\text{div}u^k|_\infty)\nonumber\\
\leq & C |\bar{\varphi}^{k+1}|_2^2+C(1+|\varphi^k\nabla^3 \mathrm{div} u^k|_2^2)|\bar{W}^{k+1}|_2^2+\frac{a_1\alpha}{10}|\varphi^{k+1}\mathrm{div}\bar{u}^{k+1}|_2^2,\nonumber\\
\mathrm{J}_{10}=&2 a_1 \alpha\frac{\delta_1-1}{\delta_1} \int \nabla(\varphi^k)^2 \cdot(Q_1(u^k)-Q_1(u^{k-1})) \cdot \bar{u}^{k+1} \nonumber\\
\leq & C|\nabla \varphi^k|_\infty|\varphi^k\nabla\bar{u}^k|_2|\bar{u}^{k+1}|_2\leq C\nu^{-1}|\bar{W}^{k+1}|_2^2+\nu|\varphi^k\nabla\bar{u}^k|_2^2, \nonumber \\
\mathrm{J}_{11}=&2 a_1 \beta\frac{\delta_2-1}{\delta_2} \int \nabla (\varphi^{k})^{2m+2} \cdot(Q_2(u^k)-Q_2( u^{k-1})) \cdot \bar{u}^{k+1} \nonumber\\
\leq & C|\nabla \varphi^k|_\infty^{2m}|\varphi^k\mathrm{div}\bar{u}^k|_2|\bar{u}^{k+1}|_2\leq C\nu^{-1}|\bar{W}^{k+1}|_2^2+\nu|\varphi^k\mathrm{div}\bar{u}^k|_2^2. \nonumber
\end{align}

Then, from \eqref{35} and \eqref{34},  it yields that 
\begin{equation}\label{36}
\begin{aligned}
&\frac{\mathrm{d}}{\mathrm{d}t} \int(\bar{W}^{k+1})^{\top} A_0 \bar{W}^{k+1}+a_1  \alpha|\varphi^{k+1} \nabla \bar{u}^{k+1}|_2^2+a_1\alpha |\varphi^{k+1}\mathrm{div}\bar{u}^{k+1}|_2^2 \\
\leq &  C\left(\nu^{-1}+|\varphi^k \nabla^4 u^k|_2^2\right)|\bar{W}^{k+1}|_2^2+C |\bar{\varphi}^{k+1}|_2^2  +\nu\left(|\varphi^k \nabla\bar{u}^k|_2^2+|\bar{\varphi}^k|_2^2+|\bar{W}^k|_2^2\right) .
\end{aligned}
\end{equation}

We denote
\begin{equation*}
  S^{k+1}(t)=\sup\limits_{s\in[0,t]}|\bar{W}^{k+1}(s)|_2^2+\sup\limits_{s\in[0,t]}|\bar{\varphi}^{k+1}(s)|_2^2.
\end{equation*}

From \eqref{37}- \eqref{36}, we have
\begin{equation}
\begin{aligned}
& \frac{\mathrm{d}}{\mathrm{d}t} \int\Big((\bar{W}^{k+1})^{\top} A_0 \bar{W}^{k+1}+|\bar{\varphi}^{k+1}(s)|_2^2\Big)+|\varphi^{k+1} \nabla \bar{u}^{k+1}|_2^2 \\
\leq &  E^k_\nu(|\bar{W}^{k+1}|_2^2+|\bar{\varphi}^{k+1}|_2^2 ) +\nu\left(|\varphi^k \nabla\bar{u}^k|_2^2+|\bar{\varphi}^k|_2^2+|\bar{W}^k|_2^2\right)
\end{aligned}
\end{equation}
for some $E^k_\nu$ such that $\int_{0}^{t} E^k_\nu ds\leq C+C(1+\frac{1}{\nu})t$. According to Gronwall's inequality, one has
\begin{equation}
\begin{aligned}
&S^{k+1}+\int_{0}^{t}|\varphi^{k+1}\nabla\bar{u}^{k+1}|_2^2\mathrm{d}s\\
\leq & C\nu \int_{0}^{t}(|\varphi^k \nabla\bar{u}^k|_2^2+|\bar{\varphi}^k|_2^2+|\bar{W}^k|_2^2)\mathrm{d}s\exp(C+C(1+\frac{1}{\nu})t)\\
\leq & \Big(C\nu\int_{0}^{t}|\varphi^k \nabla\bar{u}^k|_2^2\mathrm{d}s+Ct\nu\sup\limits_{s\in[0,t]}[|\bar{W}^k|_2^2+|\bar{\varphi}^k|_2^2]\Big)\exp(C+C(1+\frac{1}{\nu})t).
\end{aligned}
\end{equation}
We can choose $\nu_0>0$ and $T_{*}\in(0,\min(1,T^{***}))$ small enough such that 
\begin{equation*}
  C\nu_0\exp C\leq \frac{1}{8},\quad \exp(C(1+\frac{1}{\nu})T_*)\leq 2,
\end{equation*}
which yields that
\begin{equation}\label{40}
 \sum_{k=1}^{+\infty} \Big(S^{k+1}(T_*)+\int_{0}^{T_*}(|\varphi^{k+1} \nabla \bar{u}^{k+1}|_2^2)\mathrm{d}t\Big)\leq C< +\infty.
\end{equation}
It follows from \eqref{20} and \eqref{40} that $(\varphi^k, W^k)$ converges to a limit $(\varphi,W)$ in the following strong sense:
\begin{equation}
(\varphi^k, W^k)\rightarrow (\varphi, W)\quad \text{in} \quad L^\infty([0,T_*];H^2(\mathbb{R}^3)).
\end{equation}
Due to the local estimates \eqref{20} and the lower-continuity of norm for weak or weak* convergence, we know that $(\varphi,\psi, W)$ satisfies the estimates \eqref{20}. According to the strong convergence in \eqref{40}, we can  show that $(\varphi,\psi, W)$ is a weak solution of \eqref{4} in the sense of distribution with the regularities:
\begin{equation}
\begin{aligned}
& \varphi \in L^{\infty}\left(\left[0, T^*\right] ; H^3\right), \varphi_t \in L^{\infty}\left(\left[0, T^*\right] ; H^2\right), \phi \in L^{\infty}\left(\left[0, T^*\right] ; H^3\right), \phi_t \in L^{\infty}\left(\left[0, T^*\right] ; H^2\right), \\
&u \in L^{\infty}\left(\left[0, T^*\right] ; H^3\right),  \varphi \nabla^4 u \in L^2\left(\left[0, T^*\right] ; L^2\right),   u_t \in L^{\infty}\left(\left[0, T^*\right] ; H^1\right) \cap L^2\left(\left[0, T^*\right] ; D^2\right) .
\end{aligned}
\end{equation}
Thus the existence of strong solutions is proved.

Step 2. Uniqueness and time-continuity. It can be obtained via the same arguments used in the proof of Lemma \ref{85}. 
\qed

\subsection{Proof of Theorem \ref{38}}\label{89} Now we are ready to prove Theorem \ref{38}.
\begin{proof}
Step 1. Existence of regular solutions. First, for the initial assumption \eqref{66}, it follows from Theorem \ref{29} that there exists a positive time $T_*$  such that the problem \eqref{4} has a unique strong solution $(\varphi, \phi, u)$ in $\left[0, T_*\right] \times \mathbb{R}^3$ satisfying the regularities in \eqref{88}, which means that
\begin{equation*}
\left(\rho^{\frac{\delta_1-1}{2}}, \rho^{\frac{\gamma-1}{2}}\right)=(\varphi, \phi) \in C^1\left(\left(0, T_*\right) \times \mathbb{R}^3\right), \quad \text { and } \quad(u,  \nabla u) \in C\left(\left(0, T_*\right) \times \mathbb{R}^3\right) \text {. }
\end{equation*}

Because of the assumption \eqref{78}, without loss of generality, we assume that $\delta_1\leq 3.$ Noticing that $\rho=\varphi^{\frac{2}{\delta_1-1}}$ and $\frac{2}{\delta_1-1} \geq 1$, it is easy to show that
\begin{equation*}
\rho \in C^1\left(\left(0, T_*\right) \times \mathbb{R}^3\right) .
\end{equation*}
Second, the system $\eqref{4}_2$ for $W=(\phi, u)$ could be written as
\begin{equation}\label{50}
\left\{\begin{split}
&\phi_t+u \cdot \nabla \phi+\frac{\gamma-1}{2} \phi \mathrm{div} u=0, \\
&u_t+u \cdot \nabla u+\frac{A \gamma}{\gamma-1}\nabla \phi^2+\alpha\varphi^2 \triangle u+(\alpha+\beta\varphi^{2m})\varphi^2 \nabla\text{div}u \\
=&\frac{\alpha \delta_1}{\delta_1-1}Q_1(u)\cdot \nabla \varphi^2+\frac{\beta \delta_2}{\delta_2-1}Q_2(u)\cdot \nabla \varphi^{2m+2}.
\end{split}\right.
\end{equation}
Multiplying $\eqref{50}_1$ by $\frac{\partial \rho}{\partial \phi}(t, x)=\frac{2}{\gamma-1} \phi^{\frac{3-\gamma}{\gamma-1}}(t, x) \in C\left(\left(0, T_*\right) \times \mathbb{R}^3\right)$ on both sides, we get the continuity equation in $\eqref{1}_1$ :
\begin{equation*}
\rho_t+u \cdot \nabla \rho+\rho  \text{div} u=0 .
\end{equation*}

Multiplying $\eqref{50}_2$ by $\phi^{\frac{2}{\gamma-1}}=\rho(t, x) \in C^1\left(\left(0, T_*\right) \times \mathbb{R}^3\right)$ on both sides, we get the momentum equations in $\eqref{1}_2$ :
\begin{equation*}
\rho u_t+\rho u \cdot  \nabla u+ \nabla P=\text{div} \mathbb{T}.
\end{equation*}
Finally, recalling that $\rho$ can be represented by the formula
\begin{equation*}
\rho(t, x)=\rho_0(X(0, t, x)) \exp \left(\int_0^t  \text{div} u(s, X(s, t, x)) \mathrm{d} s\right),
\end{equation*}
where $X \in C^1\left(\left[0, T_*\right] \times\left[0, T_*\right] \times \mathbb{R}^3\right)$ is the solution to the initial value problem
\begin{equation*}
\left\{\begin{array}{l}
\frac{\mathrm{d}}{\mathrm{d} s} X(s, t, x)=u(s, X(s, t, x)), \quad 0 \leq s \leq T_*, \\
X(t, t, x)=x, \quad 0 \leq t \leq T_*, \quad x \in \mathbb{R}^3,
\end{array}\right.
\end{equation*}
it is obvious that
\begin{equation*}
\rho(t, x) \geq 0, \quad \forall(t, x) \in\left(0, T_*\right) \times \mathbb{R}^3.
\end{equation*}
That is to say, $(\rho, u)$ satisfies the problem \eqref{1}-\eqref{3} in the sense of distributions, and has the regularities shown in Definition \ref{87}, which means that the Cauchy problem \eqref{1}-\eqref{3}  has a unique regular solution $(\rho, u)$.

Step 2. The smoothness of regular solutions. Now we will show that the regular solution that we obtained in the above step is indeed a classical one in positive time $\left(0, T_*\right]$

Due to the definition of regular solution and the classical Sobolev imbedding theorem, we immediately know that
\begin{equation*}
\left(\rho,  \nabla \rho, \rho_t, u, \nabla u\right) \in C\left(\left[0, T_*\right] \times \mathbb{R}^3\right) .
\end{equation*}
Now we only need to prove that
\begin{equation*}
\left(u_t, \text{div} \mathbb{T} \right) \in C\left(\left[0, T_*\right] \times \mathbb{R}^3\right).
\end{equation*}
Next, we first give the continuity of $u_t$. We differentiate $\eqref{50}_2$ with respect to $t:$
\begin{equation}\label{51}
\begin{split}
&u_{tt}+\alpha(\varphi^2)_t \triangle u+\alpha\varphi^2 \triangle u_t +(\varphi^2(\alpha+\beta\varphi^{2m}))_t \nabla \text{div}u + (\alpha+\beta\varphi^{2m})\varphi^2\nabla \text{div} u_t\\
=&-(u \cdot \nabla u)_t-\frac{A \gamma}{\gamma-1}\nabla (\phi^2)_t+\frac{\alpha \delta_1}{\delta_1-1}Q_1(u_t)\cdot \nabla \varphi^2+\frac{\alpha \delta_1}{\delta_1-1}Q_1(u)\cdot \nabla (\varphi^2)_t\\
&+\frac{\beta \delta_2}{\delta_2-1}Q_2(u_t)\cdot \nabla \varphi^{2m+2}+\frac{\beta \delta_2}{\delta_2-1}Q_2(u)\cdot \nabla (\varphi^{2m+2})_t,
\end{split}
\end{equation}
which along with $\eqref{88},$ easily implies that
\begin{equation}\label{105}
u_{tt}\in L^2([0,T_*]; L^2).
\end{equation}
Applying the operator $\partial_x^\zeta( |\zeta|=2)$ to $\eqref{51},$   multiplying the resulting equations by $ \partial_x^\zeta u_{t}$ and integrating over $\mathbb{R}^3$, we have
\begin{equation}\label{52}
\begin{aligned}
&\frac{1}{2} \frac{\mathrm{d}}{\mathrm{d}t} |\partial_x^\zeta u_t|_2^2+\alpha|\varphi \nabla \partial_x^\zeta u_t|_2^2+\underline{(\alpha+\beta\varphi^{2m})}|\varphi \text{div} \partial_x^\zeta u_t|_2^2\\
=& -\alpha\int \nabla\varphi^2 \nabla \partial_x^\zeta u_t\cdot  \partial_x^\zeta u_t-\int \nabla  ((\alpha+\beta\varphi^{2m}) \varphi^2)\mathrm{div} \partial_x^\zeta u_t\cdot  \partial_x^\zeta u_t\\
&+\int \Big( - \partial_x^\zeta (u\nabla u)_t-\frac{A\gamma}{\gamma-1} \partial_x^\zeta\nabla(\phi^2)_t \Big)\cdot \partial_x^\zeta u_t\\
&+\int\Big(-\alpha \partial_x^\zeta((\varphi^2)_t\triangle u)-\partial_x^\zeta((\varphi^2(\alpha+\beta\varphi^{2m}))_t \nabla \text{div}u)\Big)\cdot \partial_x^\zeta u_t\\
&-\alpha\int\Big(\partial_x^\zeta(\varphi^2 \triangle u_t)-\varphi^2 \triangle \partial_x^\zeta u_t \Big)\cdot \partial_x^\zeta u_t\\
&-\int\Big(\partial_x^\zeta((\alpha+\beta\varphi^{2m})\varphi^2\nabla \text{div} u_t)-(\alpha+\beta\varphi^{2m})\varphi^2 \nabla \text{div}\partial_x^\zeta u_t\Big)\cdot \partial_x^\zeta u_t\triangleq \sum_{i=1}^{8} \text{P}_i.
\end{aligned}
\end{equation}
It follows from Lemmas  \ref{23}-\ref{17}, H\"older's inequality and Young's inequality that
\begin{align}
\text{P}_1&=-\alpha\int \nabla\varphi^2 \nabla \partial_x^\zeta u_t\cdot  \partial_x^\zeta u_t\nonumber\\
&\leq C|\varphi \nabla^3 u_t|_2|\nabla^2 u_t|_2|\nabla \varphi|_\infty \leq \frac{\alpha}{10}|\varphi \nabla^3 u_t|_2^2+C |\nabla^2 u_t|_2^2,\nonumber\\
\text{P}_2&= \int \nabla  ((\alpha+\beta\varphi^{2m}) \varphi^2)\mathrm{div} \partial_x^\zeta u_t\cdot  \partial_x^\zeta u_t \nonumber\\
&\leq C|\nabla \varphi|_\infty(1+|\varphi|_\infty^{2m})|\nabla^2 u_t|_2|\varphi\text{div}\nabla^2 u_t|_2 \leq \frac{\alpha}{10}|\varphi \nabla^2 \mathrm{div} u_t|_2^2+C |\nabla^2 u_t|_2^2, \nonumber\\
\text{P}_3&= -\int \partial_x^\zeta (u\nabla u)_t\cdot \partial_x^\zeta u_t \label{74}\\
&\leq C(|\nabla^3 u|_2|u_t|_\infty|\nabla^2 u_t|_2+|\nabla u|_\infty|\nabla^2 u_t|_2^2+|\nabla u_t|_3|\nabla^2 u|_6|\nabla^2 u_t|_2)\nonumber\\
&\leq C|\nabla^2 u_t|_2^2+C, \nonumber\\
\text{P}_4&=-\int \frac{A\gamma}{\gamma-1} \partial_x^\zeta\nabla(\phi^2)_t\cdot \partial_x^\zeta u_t \nonumber\\
& \leq C(|\nabla \phi|_\infty|\nabla^2\phi_t|_2|\nabla^2 u_t|_2+|\nabla \phi_t|_3|\nabla^2 \phi|_6|\nabla^2 u_t|_2) \leq C|\nabla^2 u_t|_2^2+C,\nonumber\\
\text{P}_5&=-\int\Big(\alpha \partial_x^\zeta((\varphi^2)_t\triangle u)\cdot \partial_x^\zeta u_t\nonumber\\
&\leq C(|\varphi_t|_\infty|\varphi\nabla^4u|_2|\nabla^2 u_t|_2+|\nabla^2\varphi|_3|\varphi_t|_\infty|\nabla^2u|_6|\nabla^2 u_t|_2 \nonumber\\
&+|\nabla \varphi|_\infty|\varphi_t|_\infty|\nabla^3 u|_2|\nabla^2 u_t|_2+|\nabla \varphi_t|_3|\nabla^2 u|_6|\varphi\nabla^3 u_t|_2+|\nabla\varphi_t|_3|\varphi \nabla^3 u|_6|\nabla^2 u_t|_2)\nonumber\\
&\leq C|\nabla^2 u_t|_2^2+\frac{\alpha}{10} |\varphi \nabla^3 u_t|_2^2+C|\varphi\nabla^4 u|_2^2+C,\nonumber\\
\text{P}_6&=- \int\partial_x^\zeta\Big((\varphi^2(\alpha+\beta\varphi^{2m}))_t \nabla \text{div}u\Big)\cdot \partial_x^\zeta u_t\nonumber\\
&\leq C\Big(|\varphi \nabla^3\text{div}u|_2|\nabla^2 u_t|_2(|\varphi_t|_\infty+|\varphi_t|_\infty|\varphi|_\infty^{2m}) +|\nabla^2 u_t|_2|\nabla^2 \text{div}u|_2|\nabla\varphi|_\infty|\varphi_t|_\infty(1+|\varphi|_\infty^{2m})\nonumber\\
&+|\nabla \varphi_t|_3|\varphi\nabla^2\text{div} u|_6|\nabla^2 u_t|_2(1+|\varphi|_\infty^{2m})+|\varphi\nabla^3 u_t|_2|\nabla^2 \text{div}u|_2|\varphi_t|_\infty(1+|\varphi|_\infty^{2m})\Big) \nonumber\\
&\leq C|\nabla^2 u_t|_2^2+\frac{\alpha}{10} |\varphi \nabla^2 \mathrm{div} u_t|_2^2+C|\varphi \nabla^3 \mathrm{div} u|_2^2+C,\nonumber\\
\text{P}_7&=-\alpha\int\Big(\partial_x^\zeta(\varphi^2 \triangle u_t)-\varphi^2 \triangle \partial_x^\zeta u_t \Big)\cdot \partial_x^\zeta u_t\nonumber\\
&\leq C(|\nabla\varphi|_\infty^2|\nabla^2 u_t|_2^2+|\varphi\nabla^3 u_t|_2|\nabla \varphi |_\infty |\nabla^2 u_t|_2+|\varphi\nabla^2 u_t|_6|\nabla^2 \varphi|_3|\varphi|_\infty|\nabla^2 u_t|_2)\nonumber\\
&\leq C|\nabla^2 u_t|_2^2+\frac{\alpha}{10} |\varphi \nabla^3 u_t|_2^2,\nonumber\\
\text{P}_8&=-\int\Big(\partial_x^\zeta((\alpha+\beta\varphi^{2m})\varphi^2\nabla \text{div} u_t)-(\alpha+\beta\varphi^{2m})\varphi^2 \nabla \text{div}\partial_x^\zeta u_t\Big)\cdot \partial_x^\zeta u_t\nonumber\\
&\leq C(1+|\varphi|_\infty^{2m})(|\varphi\nabla^2 \text{div}u_t|_2|\nabla \varphi|_\infty |\nabla^2 u_t|_2+|\nabla^2 \varphi|_3|\varphi\nabla\text{div}u_t|_6|\nabla^2 u_t|_2\nonumber\\
&\quad +|\nabla\varphi|_\infty^2|\nabla \text{div}u_t|_2^2)\nonumber\\
&\leq C|\nabla^2 u_t|_2^2+\frac{\alpha}{10} |\varphi \nabla^2 \mathrm{div} u_t|_2^2.\nonumber
\end{align}
Substituting \eqref{74} into $\eqref{52}$ yields
\begin{equation}\label{54}
\frac{\mathrm{d}}{\mathrm{d} t} |\nabla^2 u_t|_2^2+|\varphi \nabla^3 u_t|_2^2\leq C|\nabla^2 u_t|_2^2+C|\varphi \nabla^4u|_2^2+C.
\end{equation}

Then multiplying both sides of \eqref{54} with $t$ and integrating over $[\tau, t]$ for any $\tau \in(0, t)$, one gets
\begin{equation}\label{55}
t\left|u_t\right|_{D^2}^2+\int_\tau^t s\left|\varphi \nabla^3 u_t\right|_2^2 \mathrm{~d} s \leq C \tau\left|u_t(\tau)\right|_{D^2}^2+C(1+t) .
\end{equation}

According to the definition of the regular solution, we know that
$$
\nabla^2 u_t \in L^2\left(\left[0, T_*\right] ; L^2\right),
$$
which, along with Lemma \ref{106}, implies that there exists a sequence $s_k$ such that
$$
s_k \rightarrow 0, \quad \text { and }\quad  s_k\left|\nabla^2 u_t\left(s_k, \cdot\right)\right|_2^2 \rightarrow 0, \quad \text { as } k \rightarrow+\infty .
$$

Then, letting $\tau=s_k \rightarrow 0$ in \eqref{55}, we have
$$
t\left|u_t\right|_{D^2}^2+\int_0^t s\left|\varphi \nabla^3 u_t\right|_2^2 \mathrm{d} s  \leq C(1+t) \leq C.
$$

So we have
\begin{equation}\label{56}
t^{\frac{1}{2}} u_t \in L^{\infty}\left(\left[0, T_*\right] ; H^2\right).
\end{equation}

Based on the classical Sobolev imbedding theorem
\begin{equation}\label{57}
L^{\infty}\left([0, T] ; H^1\right) \cap W^{1,2}\left([0, T] ; H^{-1}\right) \hookrightarrow C\left([0, T] ; L^q\right),
\end{equation}

for any $q \in(3,6)$, from \eqref{105} and \eqref{56}, we have
$$
t u_t \in C\left(\left[0, T_*\right] ; W^{1,4}\right),
$$
which implies that $u_t \in C\left(\left(0, T_*\right] \times \mathbb{R}^3\right)$.
Finally, we consider the continuity of $\operatorname{div} \mathbb{T}$. Denote $\mathbb{N}= \varphi^2 \triangle u+\varphi^2(\alpha+\beta\varphi^{2m})\nabla\text{div}u - Q_1(u)\cdot \nabla \varphi^2  -Q_2(u)\cdot \nabla\varphi^{2m+2} $. Based on \eqref{88} and \eqref{56}, we have
$$
t \mathbb{N} \in L^{\infty}\left(0, T_* ; H^2\right) .
$$
From $\mathbb{N}_t \in L^2\left(0, T_* ; L^2\right)$ and \eqref{57}, we obtain $t \mathbb{N} \in C\left(\left[0, T_*\right] ; W^{1,4}\right)$, which implies that $\mathbb{N} \in C\left(\left(0, T_*\right] \times \mathbb{R}^3\right)$. Since $\rho \in C\left(\left[0, T_*\right] \times \mathbb{R}^3\right)$ and $\operatorname{div} \mathbb{T}=\rho \mathbb{N}$, then we obtain the desired conclusion.
\end{proof}

{\textbf{Acknowledgement:}} The authors sincerely appreciates Professor  Shengguo Zhu for his helpful
suggestions and discussions on the problem solved in this paper.
The research was supported in part by  National Natural Science Foundation of China under Grants 12371221, 12161141004 and 11831011. This work was also partially supported by the Fundamental Research Funds for the Central Universities and Shanghai Frontiers Science Center of Modern Analysis.

{\bf Conflict of Interest:} The authors declare that they have no conflict of interest.

\end{document}